\documentclass{amsart}
\usepackage{amssymb,amsmath,latexsym,mathrsfs,times,tikz,hyperref}

\hypersetup{colorlinks=true, linkcolor=blue, citecolor=blue}

\numberwithin{equation}{section}

\theoremstyle{plain}
\newtheorem{thm}{Theorem}[section]
\newtheorem{dfn}[thm]{Definition}
\newtheorem{cor}[thm]{Corollary}

\newtheorem{prop}[thm]{Proposition}
\newtheorem{lemma}[thm]{Lemma}

\theoremstyle{definition}
\newtheorem{ex}[thm]{Example}

\newtheorem*{sublemma1}{Lemma \ref{thm:configgivescat}.1} 
\newtheorem*{sublemma2}{Lemma \ref{thm:configgivescat}.2}

\newcommand{\del}{\backslash}
\newcommand{\PG}{\mathrm{PG}}
\newcommand{\AG}{\mathrm{AG}}
\newcommand{\GF}{\mathrm{GF}}
\newcommand{\cl}{\mathrm{cl}}

\newcommand{\mcZ}{\mathcal{Z}}

\newcommand{\frp}{\mathbin{\Box}}
\newcommand{\shuffle}{\mathbin{\wr}}

\title{The $\mathcal{G}$-invariant and catenary data of a matroid}

\author[J.~Bonin]{Joseph E.~Bonin} \address[J.~Bonin]
{Department of Mathematics\\ The George Washington University\\
  Washington, D.C. 20052, USA} \email[J.~Bonin] {jbonin@gwu.edu}
\author[J.P.S.~Kung] {Joseph P.S.~Kung} \address[J.P.S.~Kung]
{Department of Mathematics\\
  University of North Texas\\
  Denton, TX 76203, USA} \email[J.P.S.~Kung] {kung@unt.edu}
\date{\today.}


\begin{document}

\begin{abstract}    
  The catenary data of a matroid $M$ of rank $r$ on $n$ elements is
  the vector $(\nu(M;a_0,a_1,\ldots,a_r))$, indexed by compositions
  $(a_0,a_1,\ldots,a_r)$, where $a_0 \geq 0$,\, $a_i > 0$ for
  $i \geq 1$, and $a_0+ a_1 + \cdots + a_r = n$, with the coordinate
  $\nu (M;a_0,a_1, \ldots,a_r)$ equal to the number of maximal chains
  or flags $(X_0,X_1, \ldots,X_r)$ of flats or closed sets such that
  $X_i$ has rank $i$,\, $|X_0| = a_0$, and $|X_i - X_{i-1}| = a_i$. We
  show that the catenary data of $M$ contains the same information
  about $M$ as its $\mathcal{G}$-invariant, which was defined by
  H. Derksen [\emph{J.\ Algebr.\ Combin.}\ 30 (2009) 43--86].  The
  Tutte polynomial is a specialization of the $\mathcal{G}$-invariant.
  We show that many known results for the Tutte polynomial have
  analogs for the $\mathcal{G}$-invariant.  In particular, we show
  that for many matroid constructions, the $\mathcal{G}$-invariant of
  the construction can be calculated from the $\mathcal{G}$-invariants
  of the constituents and that the $\mathcal{G}$-invariant of a
  matroid can be calculated from its size, the isomorphism class of
  the lattice of cyclic flats with lattice elements labeled by the
  rank and size of the underlying set.  We also show that the number
  of flats and cyclic flats of a given rank and size can be derived
  from the $\mathcal{G}$-invariant, that the $\mathcal{G}$-invariant
  of $M$ is reconstructible from the deck of $\mathcal{G}$-invariants
  of restrictions of $M$ to its copoints, and that, apart from free
  extensions and coextensions, one can detect whether a matroid is a
  free product from its $\mathcal{G}$-invariant.

\end{abstract}

\maketitle

\markboth{\emph{The $\mathcal{G}$-invariant and catenary data of a
    matroid}}{\emph{The $\mathcal{G}$-invariant and catenary data of a
    matroid}}

\section{The $\mathcal{G}$-invariant}\label{sec:basics}

Motivated by work on the $F$-invariant by Billera, Jia, and Reiner
\cite{Billera}, Derksen introduced the $\mathcal{G}$-invariant for
matroids and polymatroids in \cite{Derksen}. Together with Fink,
Derksen showed in \cite{DerksenFink} that it is a universal valuative
invariant for subdivisions of matroid base polytopes.

Let $M$ be a rank-$r$ matroid with rank function $r$ and closure
operator $\cl$ on the set $\{1,2,\ldots,n\}$.  The \emph{rank
  sequence} $\underline{r}(\pi) = r_1 r_2 \ldots r_n$ of a permutation
$\pi$ on $\{1,2,\ldots,n\}$ is the sequence defined by
$r_1 = r(\{\pi(1)\})$ and for $j \geq 2$,
\[
r_j = r(\{\pi(1), \pi(2),\ldots,\pi(j)\}) - r(\{\pi(1),
\pi(2),\ldots,\pi(j-1)\}).
\]
In a rank sequence, $r_j$ is $0$ or $1$, there are exactly $r$ $1$'s,
and the set $\{\pi(j): r_j =1\}$ is a basis of $M$.  The rank sequence
$\underline{r}(\pi)$ is an \emph{$(n,r)$-sequence}, that is, a
sequence of $n$ terms, $r$ of which are $1$ and the other $n-r$ of
which are $0$.

Let $[\underline{r}]$ be a variable or formal symbol, one for each
$(n,r)$-sequence, and let $\mathcal{G}(n,r)$ be the vector space (over
a field of characteristic zero) of dimension $\binom {n}{r}$
consisting of all formal linear combinations of such symbols.  The
\emph{$\mathcal{G}$-invariant} and its coefficients
$g_{\underline{r}}(M)$ are defined by
\[
\mathcal{G}(M) = \sum_{\underline{r}} g_{\underline{r}}(M)
[\underline{r}] = \sum_{\pi} [\underline{r}(\pi)] 
\]
where the sum is over all $n!$ permutations of $\{1,2,\ldots,n\}$.
(This definition is essentially Derksen's.  We just replace a
quasisymmetric function constructed from $\underline{r}$ by the symbol
$[\underline{r}]$.) A \emph{specialization} of the
$\mathcal{G}$-invariant with values in an abelian group $\mathbb{A}$
is a function assigning a value in $\mathbb{A}$ to each symbol
$[\underline{r}]$.  Specializations may be given directly (as in
Theorem \ref{GtoT}) or indirectly.  In particular, if $L$ is a linear
transformation from $\mathcal{G}(n,r)$ to a vector space $V$, then the
assignment $[\underline{r}] \mapsto L([\underline{r}])$ gives a
specialization.

Recall that the Tutte polynomial $T(M;x,y)$ of a rank-$r$ matroid $M$
on the set $E$ is defined by
\[
T(M;x,y) = \sum_{A \subseteq E}  (x-1)^{r-r(A)} (y-1)^{|A| - r(A)}.
\]
Derksen \cite{Derksen} showed that the $\mathcal{G}$-invariant
specializes to the Tutte polynomial.  In the next theorem, we state
his specialization without using quasisymmetric functions.

\begin{thm}[Derksen]\label{GtoT}
  Let $M$ be a rank-$r$ matroid on a set of $n$ elements.  Then the
  Tutte polynomial $T(M;x,y)$ can be obtained from $\mathcal{G}(M)$ by
  the specialization
  \[ 
  [\underline{r}] \mapsto \sum_{m=0}^n \frac {(x-1)^{r-\mathrm{wt}(r_1
      r_2 \ldots r_m)}(y-1)^{m-\mathrm{wt}(r_1 r_2 \ldots r_m)}}{m!
    (n-m)!},
  \]
  where $\mathrm{wt}(r_1 r_2 \ldots r_m)$ is the number of $1$'s in
  the initial segment $r_1 r_2 \ldots r_m$ of $\underline{r}$.
\end{thm}  

In this paper, we study the $\mathcal{G}$-invariant from a
combinatorial point of view.  After some order-theoretic preliminaries
in Section \ref{sec:PO}, we begin in Section \ref{sec:catenarydata} by
determining the exact combinatorial information contained in the
$\mathcal{G}$-invariant.  This information is encoded in the catenary
data, a vector or array of integers that records the number of flags
or maximal chains of flats with given sizes.  The
$\mathcal{G}$-invariant and catenary data contain the same
information; indeed, the catenary data are the coefficients of the
$\mathcal{G}$-invariant when expanded in a new basis of
$\mathcal{G}(n,r)$ called the $\gamma$-basis.  In Section
\ref{sec:constructions}, we study how the $\mathcal{G}$-invariant and
catenary data behave under matroid constructions.  Many of the known
results saying that the Tutte polynomial of a construction can be
calculated from the Tutte polynomials of its constituents have
counterparts for the $\mathcal{G}$-invariant.  We show in Section
\ref{sec:parameters} that many parameters of a matroid not derivable
from the Tutte polynomial can be derived from the
$\mathcal{G}$-invariant.  In addition, we show, in Section
\ref{sec:recon}, that the $\mathcal{G}$-invariant can be reconstructed
from several decks, giving analogs of reconstructibility results for
the Tutte polynomial.  In Section \ref{sec:configurations}, extending
a result of Eberhardt \cite{jens} for the Tutte polynomial, we show
that the $\mathcal{G}$-invariant is determined by the isomorphism type
of the lattice of cyclic flats, along with the size and rank of the
cyclic flat that corresponds to each element in the lattice.  In the
final section, we show that, except for free extensions and
coextensions, whether a matroid is a free product can be detected from
its $\mathcal{G}$-invariant.

Although we focus on matroids, we remark that the
$\mathcal{G}$-invariant of matroids constructed from graphs, such as
cycle or bicircular matroids, should have many applications in graph
theory.  For example, whether a graph has a Hamiltonian cycle is not
deducible from the Tutte polynomial of the graph but is deducible from
the $\mathcal{G}$-invariant of its cycle matroid (Corollary
\ref{cocircuits}).

\section{Partial orders on sequences and compositions}\label{sec:PO}

We use the notation $0^a$ for a sequence of $a$ $0$'s and $1^b$ for a
sequence of $b$ $1$'s.  An \emph{$(n,r)$-composition} is a
length-$(r+1)$ integer sequence $(a_0, a_1, a_2,\ldots, a_r)$
satisfying the inequalities $a_0 \geq 0$ and $a_j > 0$ for
$1 \leq j \leq r$, together with the equality
\[
a_0 + a_1 + a_2 + \cdots + a_r = n.
\]
The correspondence 
\[
0^{a_0}10^{a_1 - 1}10^{a_2 - 1} \ldots 10^{a_r - 1}
\,\,\longleftrightarrow\,\, (a_0, a_1, a_2,\ldots, a_r)
\]
gives a bijection between the set of $(n,r)$-sequences and the set of
$(n,r)$-compositions.  In this paper, we identify an $(n,r)$-sequence
and its corresponding $(n,r)$-composition.

We also need a partial order on $(n,r)$-sequences.  If $\underline{s}$
and $\underline{t}$ are $(n,r)$-sequences, then $\underline{t}
\trianglerighteq \underline{s}$ if
\begin{equation*}\label{ineq:st}
  t_1 + t_2 + \cdots + t_j  \geq  s_1 + s_2 + \cdots + s_j 
\end{equation*}
for every index $j$; in other words, reading from the left, there are
always at least as many $1$'s in $\underline{t}$ as there are in
$\underline{s}$. This order has maximum $1^r0^{n-r}$ and minimum
$0^{n-r}1^r$.  Under the bijection, $\trianglerighteq$ defines a
partial order on $(n,r)$-compositions given by $(b_0,b_1,\ldots,b_r)
\trianglerighteq (a_0,a_1,\ldots,a_r)$ if and only if for every index
$j$,
\begin{equation*}  
  b_0 + b_1 + \cdots + b_j \leq a_0 + a_1 +  \cdots + a_j.
\end{equation*}
This order is a suborder of the reversed dominance order on
compositions.  While this link plays no role in what we do, we note
that the partial order $\trianglerighteq$ on $(n,r)$-sequences is a
distributive lattice that is isomorphic to a sublattice of Young's
partition lattice (see \cite[p.~288]{EC2}).

\section{Catenary data}\label{sec:catenarydata}

Let $M$ be a rank-$r$ matroid on a set $E$ of size $n$. A \emph{flag}
$(X_i) = (X_0,X_1,\ldots,X_r)$ is a maximal chain
\[ 
\cl (\emptyset) = X_0 \subset X_1 \subset X_2 \subset \cdots \subset
X_{r-1} \subset X_r = E,
\]
where $X_j$ is a rank-$j$ flat of $M$.  The \emph{composition}
$\mathrm{comp}((X_i))$ of the flag $(X_i)$ is the $(n,r)$-composition
$(a_0, a_1, a_2,\ldots, a_r)$, where $a_0 = |X_0|$, and for positive
$j$,
\[
 a_j = |X_j  -  X_{j-1}|.  
\]
Thus, $a_0$ is the number of loops in $M$.  Also, the rank sequence
$0^{a_0}1 0^{a_1-1} \ldots 10^{a_r-1}$ corresponding to
$\mathrm{comp}((X_i))$ is the rank sequence of each of the $a_0!a_1!
\cdots a_r!$ permutations that, for all positive $j$, puts the
elements of $X_{j-1}$ before those of the set-difference
$X_j-X_{j-1}$.

We use the notation $(t)_k$ for the falling factorial
$t(t-1)(t-2)\cdots(t-k+1)$, the number of sequences of $k$ distinct
objects chosen from a set of $t$ objects.  For an $(n,r)$-sequence
$\underline{s}$ with composition $(a_0,a_1,\ldots,a_r)$, let
$\gamma(a_0,a_1,\ldots,a_r)$, and its coefficients
$c_{(a_0,a_1,\ldots,a_r)}(b_0,b_1,\ldots,b_r)$ be defined by
\begin{eqnarray*}
  && \gamma (a_0,a_1, \ldots,a_r) =  \sum_{(b_0,b_1,\ldots,b_r)}
  c_{(a_0,a_1,\ldots,a_r)}(b_0,b_1,\ldots,b_r) [0^{b_0}10^{b_1-1}
  \ldots 10^{b_r-1}] \\
  &&  =  \sum_{(b_j) \trianglerighteq (a_j)} (a_0)_{b_0}
  \Biggl(\prod_{j=1}^r a_j   \biggl(a_j-1+ \Bigl(\sum_{i=0}^{j-1}
  a_i-b_i\Bigr)
  \biggr)_{b_j-1} \Biggr)\,[0^{b_0}10^{b_1-1} \ldots 10^{b_r-1}],
\end{eqnarray*}
where the second sum ranges over all compositions with
$(b_0,b_1,\ldots,b_r) \trianglerighteq (a_0,a_1, \ldots,a_r)$.  For
example,
\begin{eqnarray*}
  && \gamma(0,1,1,4) = 24\,[111000], \quad 
  \\
  && \gamma(0,1,2,3) = 36\,[111000]+12\, [110100],  
  \\
  && \gamma(0,1,3,2) = 36\,[111000]+  24\,[110100] + 12\, [110010] ,  
  \\ 
  && \gamma(0,1,4,1) = 24\,[111000] + 24\,[110100] + 24\,[110010] 
  +24\,[110001].  
\end{eqnarray*}
Note that, by definition, the coefficient $c_{(a_0,a_1,
  \ldots,a_r)}(b_0,b_1,\ldots,b_r)$ is non-zero if and only if
$(b_0,b_1,\ldots,b_r) \trianglerighteq (a_0,a_1, \ldots,a_r)$. Hence
the matrix $(c_{(a_0,a_1, \ldots,a_r)}(b_0,b_1,\ldots,b_r))$ is a
lower triangular matrix.  In particular, the linear combinations
$\gamma(a_0,a_1, \ldots,a_r)$, where $(a_0,a_1, \ldots,a_r)$ ranges
over all $(n,r)$-compositions, form a basis, called the
\emph{$\gamma$-basis}, for the vector space $\mathcal{G}(n,r)$.

With each permutation $\pi$, we get an ordered basis
$(\pi(i_1),\pi(i_2), \ldots, \pi(i_r))$ from the terms in its rank
sequence $r_1r_2\ldots r_n$ with $r_{i_j} = 1$, where $i_1 < i_2 <
\cdots < i_r$.  This basis determines a flag, denoted
$\mathrm{flag}(\pi)$, obtained by setting $X_0 = \cl (\emptyset)$ and
\[
X_j = \cl (\{\pi(i_1),\pi(i_2), \ldots, \pi(i_j)\}).
\]   

\begin{lemma}\label{gammabasis} Let $(X_i)$ be a flag of $M$
  with composition $(a_0,a_1,\ldots,a_r)$. Then
  \[
  \sum_{\pi\,:\,\mathrm{flag}(\pi) = (X_i)} [\underline{r}(\pi)] =
  \gamma (a_0,a_1,\ldots,a_r).
  \]
  In particular, the sum on the left hand side depends only on
  $\mathrm{comp}((X_i))$.
\end{lemma}

\begin{proof} 
  Let $\pi$ be a permutation whose rank sequence $\underline{r}(\pi)$
  equals $0^{b_0} 10^{b_1-1}10^{b_2 -1} \ldots 10^{b_r-1}$. Then, by
  the Mac\,Lane-Steinitz exchange property for closure,
  $\mathrm{flag}(\pi) = (X_i)$ if and only if
 \begin{eqnarray*}
   &&  \pi(1),\pi(2), \ldots, \pi(b_0) \in X_0, \,\,\,   \\ 
   &&  \pi(b_0+1) \in  X_1 - X_0, \,\, 
   \\
   &&     \pi(b_0 + 2),  \ldots, \pi(b_0+ b_1) \in X_1,   \\ 
   &&     \quad  \vdots     \\
   &&   \pi(b_0+b_1+ \cdots + b_{j-1} + 1) \in X_j - X_{j-1},  \,\,
   \\
   && \pi(b_0 + b_1 + \cdots + b_{j-1} + 2), \ldots,
   \pi(b_0+b_1+ \cdots + b_{j-1} + b_j) \in X_j,     \\ 
   &&\quad  \vdots       \\ 
   &&   \pi(b_0+b_1+ \cdots + b_{r-1} + 1) \in X_r - X_{r-1}, \,\,  
   \\
   &&  \pi(b_0+b_1+ \cdots + b_{r-1} + 2),
   \ldots, \pi(n) \in X_r.   
  \end{eqnarray*}  
  Such a permutation $\pi$ exists if and only if
  $(b_0, b_1, \ldots, b_r) \trianglerighteq (a_0,a_1,\ldots,a_r)$ and
  all such permutations can be found by choosing a length-$b_0$
  sequence of distinct elements from $X_0$ (a set of size $a_0$), an
  element from $X_1 - X_0$ (a set of size $a_1$), a length-$(b_1 - 1)$
  sequence of distinct elements from the subset of elements in
  $X_1 \cup X_0$ not chosen earlier (a set of size
  $(a_1 - 1) + (a_0 - b_0)$), an element from $X_2 - X_1$ (a set of
  size $a_2$), a length-$(b_2-1)$ sequence of distinct elements from
  the subset of elements in $X_2 \cup X_1 \cup X_0$ not chosen earlier
  (a set of size $(a_2 - 1) + (a_1 + a_0) - (b_1 + b_0)$), and so on.
  Thus there are $c_{(a_0,a_1,\ldots,a_r)} (b_0,b_1, \ldots,b_r)$ such
  permutations $\pi$.
\end{proof}  

The ideas in the proof of Lemma \ref{gammabasis} give the next result.

\begin{lemma}\label{lem:countbybases}
  For a flag $(X_i)$ of $M$ with composition $(a_0,a_1,\ldots,a_r)$,
  there are $a_1a_2\cdots a_r$ ordered bases $(e_1,e_2,\ldots,e_r)$ of
  $M$ with $\cl(\{e_1,e_2,\ldots,e_i\})=X_i$ for $1\leq i\leq r$.
\end{lemma}

For an $(n,r)$-composition $(a_0,a_1,\ldots,a_r)$, let
$\nu(M;a_0,a_1,\ldots,a_r)$ be the number of flags $(X_i)$ in $M$ with
composition $(a_0,a_1, \ldots,a_r)$.  The \emph{catenary data} of $M$
is the $\binom {n}{r}$-dimensional vector
$(\nu(M;a_0,a_1,\ldots,a_r))$ indexed by $(n,r)$-compositions.  When
giving the catenary data of a specific matroid, we usually give only
the coordinates that might be non-zero; when a coordinate is not
given, it is zero.

\begin{thm}\label{thm:catdata}
  The $\mathcal{G}$-invariant of $M$ is determined by its catenary
  data and conversely.  In particular,
  \begin{equation}\label{gamma1}
    \mathcal{G}(M)  =  \sum_{(a_0,a_1,\ldots,a_r)}
    \nu(M;a_0,a_1,\ldots,a_r) \, \gamma(a_0,a_1,\ldots,a_r). 
  \end{equation} 
\end{thm}

\begin{proof}
  Partitioning the permutations $\pi$ on $E$ according to their flags,
  using Lemma~\ref{gammabasis}, and then partitioning the flags
  according to their compositions, we have
  \begin{align*}
    \mathcal{G}(M)  = &\, \sum_{(X_i)} \Biggl(
      \sum_{\mathrm{flag}(\pi) = (X_i)}
      [\underline{r}(\pi)]\Biggr)  \\
    =&\, \sum_{(X_i)} \gamma (|X_0|, |X_1 - X_0|, \ldots, |X_r -
    X_{r-1}|) \\   
     = &\, \sum_{(a_0,a_1,\ldots,a_r)}
    \nu(M;a_0,a_1,\ldots,a_r) \, \gamma(a_0,a_1,\ldots,a_r). \qedhere
  \end{align*}
\end{proof}

Since equation (\ref{gamma1}) is obtained by partitioning terms in a
sum (with no cancellation), writing the $\mathcal{G}$-invariant in the
symbol basis requires at least as many terms as writing it in the
$\gamma$-basis.  The next example shows that the $\gamma$-basis may
require far fewer terms than the symbol basis.

\begin{ex}
  A perfect matroid design (Young and Edmonds \cite{pmd}; see also
  Welsh \cite[Section 12.5]{Welsh}) is a matroid in which flats of the
  same rank have the same size.  Thus, for such a matroid $M$, there
  is a sequence $(\alpha_0,\alpha_1,\ldots, \alpha_r)$, where
  $r=r(M)$, such that $|F|=\alpha_i$ for all flats $F$ of rank $i$ in
  $M$.  For example, the rank-$r$ projective geometry $\PG(r-1,q)$
  over the finite field $\GF(q)$ is a perfect matroid design with the
  sequence of flat sizes being
  $$(0,1,q+1,q^2+q+1,\ldots,q^{r-1}+q^{r-2}+\cdots+q+1).$$ Another
  example is the rank-$r$ affine geometry $\AG(r-1,q)$ over $\GF(q)$,
  with its sequence being $(0,1,q,q^2,\ldots,q^{r-1})$.  If a rank-$i$
  flat $X$ is contained in exactly $t$ flats of rank $i+1$, say $X_j$,
  for $1\leq j\leq t$, then the $t$ differences $X_j-X$ partition the
  set of elements not in $X$, and $|X_j-X|= \alpha_{i+1}-\alpha_i$, so
  $t$ is $(\alpha_r-\alpha_i)/(\alpha_{i+1}-\alpha_i)$.  Therefore
  $$ \nu(M;\alpha_0,\alpha_1-\alpha_0,\alpha_2-\alpha_1,\ldots,
  \alpha_r-\alpha_{r-1}) = \prod_{i=0}^{r-1}
  \frac{\alpha_r-\alpha_i}{\alpha_{i+1}-\alpha_i},$$
  and this accounts for all flags in $M$.  (Indeed, a matroid is a
  perfect matroid design if and only if exactly one coefficient
  $\nu(M;a_0,a_1,\ldots,a_r)$ is nonzero.)  Thus,
  \begin{equation*}
    \mathcal{G}(M) = \left(\prod_{i=0}^{r-1} 
      \frac{\alpha_r-\alpha_i}{\alpha_{i+1}-\alpha_i}\right)
    \gamma(\alpha_0,\alpha_1-\alpha_0,\alpha_2-\alpha_1,
    \ldots,\alpha_r-\alpha_{r-1}). 
  \end{equation*}
  In particular,
  \[
  \mathcal{G}\bigl(\PG(r-1,q)\bigr) = \left(\prod_{i=0}^{r-1} \frac
    {q^{r-i}- 1}{q-1}\right) \gamma(0,1,q,q^2,\ldots,q^{r-1}),
  \]
  and
  $$\mathcal{G}\bigl(\AG(r-1,q)\bigr) = q^{r-1}\left(\prod_{i=1}^{r-1} \frac
    {q^{r-i}- 1}{q-1}\right)
  \gamma(0,1,q-1,q^2-q,\ldots,q^{r-1}-q^{r-2}).$$
\end{ex}

\begin{ex}\label{MN} 
  The matroids $M$ and $N$ in Figure \ref{fig:sameconfig}, the two
  smallest matroids that have the same Tutte polynomial \cite{decomp},
  also have the same catenary data, namely,
  \[
  \nu(\,\cdot\,;0,1,2,3) = 6, \qquad \nu(\,\cdot\,;0,1,1,4) = 18.
  \]
  This agrees with Derksen's calculation in \cite{Derksen} that their
  $\mathcal{G}$-invariant equals
  \[
  72[110100] + 648[111000].
  \] 
\end{ex}

\begin{figure}
  \centering
  \begin{tikzpicture}[scale=1]
  \filldraw (0,1) node {} circle  (2pt);
  \filldraw (1,1) node {} circle  (2pt);
  \filldraw (2,1) node {} circle  (2pt);
  \filldraw (0,0) node {} circle  (2pt);
  \filldraw (1,0) node {} circle  (2pt);
  \filldraw (2,0) node {} circle  (2pt);
  \draw[thick](0,0)--(2,0);
  \draw[thick](0,1)--(2,1);
  \node at (1,-0.75) {$M$};
  \filldraw (3,0.5) node {} circle  (2pt);
  \filldraw (4,0.75) node {} circle  (2pt);
  \filldraw (5,1) node {} circle  (2pt);
  \filldraw (4,0.25) node {} circle  (2pt);
  \filldraw (5,0) node {} circle  (2pt);
  \filldraw (5.5,0.5) node[right] {$x$} circle  (2pt);
  \draw[thick](3,0.5)--(5,0);
  \draw[thick](3,0.5)--(5,1);
  \node at (4,-0.75) {$N$};
  \end{tikzpicture}
  \caption{The smallest pair of matroids with the same
    $\mathcal{G}$-invariant and catenary data. }
  \label{fig:sameconfig}
\end{figure}
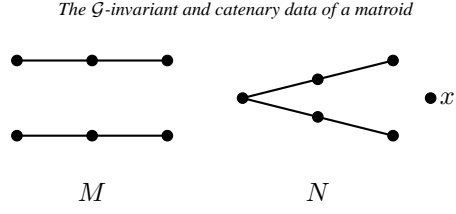

The following proposition (whose proof is immediate) gives a recursion
on restrictions to copoints (that is, rank-$(r-1)$ flats) for catenary
data.

\begin{prop}\label{catcurse}  
  Let $(a_0,a_1,\ldots,a_r)$ be an $(n,r)$-composition.  Then
  \[
  \nu(M;a_0,a_1, \ldots, a_{r-1},a_r) = \sum_{X
    \,\mathrm{a}\,\mathrm{copoint}, \, |E - X| = a_r} \nu (M|X;
  a_0,a_1,\ldots, a_{r-1}).
  \]
\end{prop}

Proposition \ref{catcurse} allows us to compute the catenary data of
paving matroids (see also \cite{FalkKung}).  Recall that a rank-$r$
matroid $M$ is \emph{paving} if all circuits have $r$ or $r+1$
elements.  Note that a matroid $M$ is paving if and only if every
symbol $[\underline{r}]$ that occurs with non-zero coefficient in its
$\mathcal{G}$-invariant starts with $r-1$ $1$'s.

\begin{thm}\label{paving} 
  Let $M$ be a rank-$r$ paving matroid on $\{1,2,\ldots,n\}$ with
  $f_{r-1}(m)$ copoints with $m$ elements.  Then $M$ has catenary
  data:
  \[
\nu(M;0,\underbrace{1,1,,\ldots, 1}_{r-2},m-r+2,n-m) = f_{r-1} (m)
  \cdot (m)_{r-2}.  
  \]
\end{thm}

All simple rank-$3$ matroids are paving, so we have explicit formulas
for their $\mathcal{G}$-invariants and catenary data.  This includes
the matroids in Example \ref{MN}.

\begin{ex}\label{Kfour} 
  The cycle matroid $M(K_4)$ of the complete graph $K_4$, has three
  $2$-point lines and four $3$-point lines; hence
  \[
  \nu(M(K_4);0,1,1,4) = 6, \qquad \nu(M(K_4);0,1,2,3) = 12.
  \]
  and
  \[
  \mathcal{G}(M(K_4)) = 6\,\gamma(0,1,1,4) + 12\,\gamma(0,1,2,3) =
  576\, [111000] + 144\, [110100].
  \]
\end{ex} 

We end this section with two more examples, the matroids $M_1$ and
$M_2$ in Figure \ref{fig:2exp}.  Brylawski \cite[p.~268]{decomp}
showed that they are the smallest pair of non-paving matroids with the
same Tutte polynomial, but Derksen \cite{Derksen} showed that their
$\mathcal{G}$-invariants are distinct.  The catenary data of these two
matroids are given in the table
\begin{center}
  \begin{tabular}{c|c|c|}
    \multicolumn{1}{r}{} & \multicolumn{1}{c}{$M_1$} &
    \multicolumn{1}{c}{$M_2$} \\
    \cline{2-3} $\nu(M_i;0,1,1,5)$ & $4$ & $6$  \rule{0pt}{10pt} \\ 
    \cline{2-3} $\nu(M_i;0,1,2,4)$ & $7$ & $3$  \rule{0pt}{10pt} \\ 
    \cline{2-3} $\nu(M_i;0,1,3,3)$ & $4$ & $6$  \rule{0pt}{10pt} \\ 
    \cline{2-3} $\nu(M_i;0,2,1,4)$ & $1$ & $3$  \rule{0pt}{10pt} \\ 
    \cline{2-3} $\nu(M_i;0,2,2,3)$ & $2$ & $1$  \rule{0pt}{10pt} \\ 
    \cline{2-3}
  \end{tabular}
\end{center}
and they are indeed different.  

\begin{figure}
 \begin{tikzpicture}[scale=1]
  \draw[thick](0,0.5)--(2,-0.2);
  \draw[thick](0,0.5)--(2,1.2);
  \draw[thick](1,0.15)--(2,1.2);
  \draw[thick](1,0.85)--(2,-0.2);
  \filldraw (1,0.85) circle  (2.5pt);
  \filldraw (2,1.1) circle  (2.5pt);
  \filldraw (1.95,1.26) circle  (2.5pt);
  \filldraw (2,-0.2) circle  (2.5pt);
  \filldraw (1,0.15) circle  (2.5pt);
  \filldraw (1.33,0.5) circle  (2.5pt);
  \filldraw (0,0.5) circle  (2.5pt);
  \node at (1,-0.75) {$M_1$};
  \end{tikzpicture}
\hspace{30pt}
 \begin{tikzpicture}[scale=1]
  \draw[thick](0,0.5)--(2,-0.2);
  \draw[thick](0,0.5)--(2,1.2);
  \filldraw (1,0.85) circle  (2.5pt);
  \filldraw (2,1.1) circle  (2.5pt);
  \filldraw (1.95,1.26) circle  (2.5pt);
  \filldraw (2,-0.2) circle  (2.5pt);
  \filldraw (0.7,0.251) circle  (2.5pt);
  \filldraw (1.4,0.017) circle  (2.5pt);
  \filldraw (0,0.5) circle  (2.5pt);
  \node at (1,-0.75) {$M_2$};
  \end{tikzpicture}
  \caption{The matroids $M_1$ and $M_2$.} 
  \label{fig:2exp}
\end{figure}
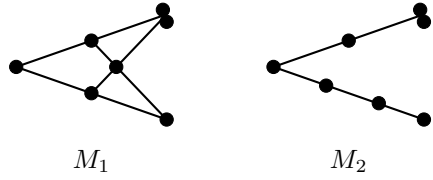

\section{Constructions on matroids}\label{sec:constructions}

In many cases, the $\mathcal{G}$-invariant of a construction is
determined by the $\mathcal{G}$-invariants of its constituents.  In
this section, we give an account of currently known results of this
type, all of which originated from results on Tutte polynomials.
Specifically, we discuss deletion, contraction, dual, truncation,
Higgs lift, direct sum, free extension, free coextension, free
product, $q$-cone, and circuit-hyperplane relaxation.

\subsection{Deletion and contraction} 
The $\mathcal{G}$-invariant does not satisfy a deletion-contraction
rule since it is strictly stronger than the Tutte polynomial, which is
universal for deletion-contraction invariants.  Still, we get a
formula involving all single-element deletions, and likewise for
contractions.  First, for two $0,1$-sequences $\underline{r}$ and
$\underline{r}'$, define the \emph{concatenation} $\diamond$ of the
corresponding symbols $[\underline{r}]$ and $[\underline{r}']$ by
$[\underline{r}] \diamond [\underline{r}'] = [\underline{r}
\underline{r}']$
and extend $\diamond$ by bilinearity.  Putting the permutations into
two groups according to whether the last element is a coloop of $M$
gives
$$\mathcal{G}(M)  = \sum\limits_{\substack{a\in E,\\ \text{not a
      coloop}}}\mathcal{G}(M\del a) \diamond[0] +
\sum\limits_{\substack{a\in E, \\ \text{a coloop} }}\mathcal{G}(M\del
a)\diamond [1].$$ Similarly,
$$\mathcal{G}(M)  = \sum\limits_{\substack{a\in E,\\ \text{not a loop}}}
[1] \diamond \mathcal{G}(M/ a) + \sum\limits_{\substack{a\in E, \\
    \text{a loop}}}[0] \diamond \mathcal{G}(M/ a).$$

\subsection{Dual} 
The next proposition, due to Derksen \cite{Derksen}, is included for
completeness.

\begin{prop}\label{prop:dual}
  The $\mathcal{G}$-invariant $\mathcal{G}(M^*)$ of the dual of $M$
  can be obtained from $\mathcal{G}(M)$ by the specialization
  \[ [r_1 r_2 \ldots r_{n-1}r_n] \mapsto [\bar{r}_n
   \bar{r}_{n-1} \ldots \bar{r}_2 \bar{r}_1],
  \]
  where $\bar{\,}$ switches $0$'s and $1$'s, that is, $\bar{0} = 1$
  and $\bar{1} = 0$.
\end{prop}

\noindent 
There is no simple version of Proposition \ref{prop:dual} for the
$\gamma$-basis.

\subsection{Truncation and Higgs lift}

The \emph{truncation} $\mathrm{Trun}(M)$ of a matroid $M$ of positive
rank is the matroid on the same set whose bases are the independent
sets of $M$ that have size $r(M)-1$.  The construction dual to
truncation is the \emph {free or Higgs lift} \cite{Higgs} defined by
$\mathrm{Lift}(M) = (\mathrm{Trun}(M^*))^*$ for a matroid $M$ with at
least one circuit.

\begin{prop}\label{prop:trun}
  Let $M$ be a matroid having positive rank.  Then
  $\mathcal{G}(\mathrm{Trun}(M))$ can be obtained from
  $\mathcal{G}(M)$ by the specialization $[\underline{r}] \mapsto
  [\underline{r}^{\downarrow}]$, where $\underline{r}^{\downarrow}$ is
  the sequence obtained from $\underline{r}$ by replacing the
  right-most $1$ by $0$. Expressed in the $\gamma$-basis,
  $\mathcal{G}(\mathrm{Trun}(M))$ can be obtained by the
  specialization given by the linear transformation $\mathcal{G}(n,r)
  \to \mathcal{G}(n,r-1)$ defined on the $\gamma$-basis by
  $\gamma(a_0,a_1,\ldots,a_r) \mapsto
  \gamma(a_0,a_1,\ldots,a_{r-1}+a_r)$.

  Let $M$ have at least one circuit.  Then
  $\mathcal{G}(\mathrm{Lift}(M))$ can be obtained from
  $\mathcal{G}(M)$ by the specialization $[\underline{r}] \mapsto
  [\underline{r}^{\uparrow}]$, where $\underline{r}^{\uparrow}$ is
  obtained from $\underline{r}$ by replacing the left-most $0$ by $1$.
\end{prop}

\subsection{Direct sum}\label{sect:dirsum} 

We first define a shuffle.  Let
$(\alpha_1,\alpha_2, \ldots, \alpha_m)$ and
$(\beta_1,\beta_2, \ldots, \beta_n)$ be sequences and $P$ be a subset
of $\{1,2,\ldots, m+n \}$ of size $m$. The \emph{shuffle}
$ \mathrm{sh}((\alpha_i),(\beta_j);P) $ is the length-$(m+n)$ sequence
obtained by inserting the sequence $(\alpha_i)$ in order into the
positions in $P$ and the sequence $(\beta_j)$ in order into the
remaining positions.  For example,
\[\mathrm{sh}((\alpha_1,\alpha_2,\alpha_3,\alpha_4),
(\beta_1,\beta_2,\beta_3);\{1,3,4,7\}) = (\alpha_1,\beta_1, \alpha_2,
\alpha_3, \beta_2, \beta_3, \alpha_4).
\]
If $[\underline{r}]$ and $[\underline{s}]$ are symbols, the first in
$\mathcal{G}(n_1,r_1)$ and the second in $\mathcal{G}(n_2,r_2)$, then
their \emph{shuffle product} $[\underline{r}]\shuffle [\underline{s}]$
is the following linear combination in
$\mathcal{G}(n_1+n_2, r_1+r_2)$:
\[ [\underline{r}] \shuffle [\underline{s}] = \sum_{P \subseteq \{1,2,
  \ldots, n_1+n_2\}, |P| = n_1} [\mathrm{sh}(\underline{r},
\underline{s};P)].
\]
The shuffle product $\shuffle$ is extended to $\mathcal{G}(n_1,r_1)
\times \mathcal{G}(n_1,r_1)$ by bilinearity.

The \emph{direct sum} $M_1 \oplus M_2$ of the matroids $M_i$ with rank
and rank functions $r_i$ on disjoint sets $E_i$ is the matroid on the
set $E_1 \cup E_2$ with rank function $r_{M_1 \oplus M_2}$ where, for
$X_i \subseteq E_i$,
\begin{equation}\label{eqn:dirsum}
  r_{M_1 \oplus M_2}
 (X_1 \cup X_2) = r_1(X_1) + r_2(X_2).  
\end{equation} 

From the fact that the rank sequence of a shuffle of two permutations
is the shuffle of their rank sequences, we obtain the following
result.

\begin{prop}\label{prop:Gdirect} 
  The $\mathcal{G}$-invariant of the direct sum $M_1 \oplus M_2$ is
  given by the formula
  \[
  \mathcal{G}(M_1 \oplus M_2) = \mathcal{G}(M_1) \shuffle
  \mathcal{G}(M_2).
  \]
\end{prop} 

The next proposition follows easily from the fact that the lattice of
flats $L(M_1 \oplus M_2)$ of the direct sum is the direct product
$L(M_1) \times L(M_2)$ of the lattices of flats of the summands,
together with formula (\ref{eqn:dirsum}).

\begin{prop}\label{prop:dirsum}
  The catenary data of the direct sum $M_1\oplus M_2$ can be
  calculated in the following way: for an
  $(|E_1|+|E_2|,r_1+r_2)$-composition
  $(a_0 + b_0, c_1,c_2,\ldots,c_{r_1+r_2})$, where $a_0$ and $b_0$ are
  the numbers of loops in $M_1$ and $M_2$, respectively, we have
  \[
  \nu(M_1 \oplus M_2; a_0 + b_0, c_1,\ldots,c_{r_1+r_2})
  =\sum_{((a_i),(b_j),P)} \nu(M_1;a_0,a_1, \ldots,a_{r_1} ) \, \nu(M_2;
  b_0,b_1, \ldots, b_{r_2})
  \]
  where the sum is over all triples
  $((a_1, \ldots,a_{r_1}), (b_1, \ldots, b_{r_2}),P)$ with $|P|=r_1$
  and
  $$\mathrm{sh}((a_1, \ldots,a_{r_1}),(b_1, \ldots, b_{r_2});P) =
  (c_1,c_2,\ldots,c_{r_1+r_2}).$$
\end{prop}

The next result treats the effect on the $\mathcal{G}$-invariant of
adding a coloop or loop.

\begin{cor}\label{cor:coloops}
  The $\mathcal{G}$-invariant $\mathcal{G}(M \oplus U_{1,1})$ can be
  obtained from $\mathcal{G}(M)$ by the specialization
  \[ [r_1r_2 \ldots r_n] \mapsto [1r_1r_2 \ldots r_n] + \sum_{j=1}^n
  \,\, [r_1 \ldots r_j 1 \ldots r_n]
  \]
  and $\mathcal{G}(M \oplus U_{0,1})$ can be obtained from
  $\mathcal{G}(M)$ by the specialization
  \[ [r_1r_2 \ldots r_n] \mapsto [0r_1r_2 \ldots r_n] + \sum_{j=1}^n
  \,\, [r_1 \ldots r_j 0 \ldots r_n].
  \]
\end{cor}

There is also a simple description of the effect of adding or removing
loops from a matroid using the $\gamma$-basis.

\begin{prop}\label{prop:manyloops}
  Let $M$ be a loopless rank-$r$ matroid.  Then when expressed in the
  $\gamma$-basis, the $\mathcal{G}$-invariants $\mathcal{G}(M \oplus
  U_{0,h})$ and $\mathcal{G}(M)$ can be obtained from each other by
  the specializations
  \[
  \gamma(h,a_1,a_2, \ldots,a_r) \longleftrightarrow \gamma(0,a_1,a_2,
  \ldots,a_r).
  \]
\end{prop}

While $M$ has exactly $a_0$ loops if and only if some term of the form
$\nu(M;a_0,a_1, \ldots,a_r)$ is positive, the number of coloops is the
greatest integer $k$ for which some composition ending in $k$ ones has
$\nu(M;a_0,a_1,\ldots,1,1)>0$.

In the cycle matroid $M(K_{r+1})$ of the complete graph, there are
$2^{r}-1$ copoints.  For $1 \leq m < \lfloor (r+1)/2 \rfloor$, there
are $\binom {r+1}{m}$ copoints isomorphic to $M(K_{r+1-m}) \oplus
M(K_m)$, with an additional $\frac {1}{2} \binom {r+1}{(r+1)/2}$
copoints isomorphic to $M(K_{(r+1)/2}) \oplus M(K_{(r+1)/2})$ if $r+1$
is even.  Hence, by Propositions \ref{catcurse} and \ref{prop:dirsum},
the catenary data of $M(K_{r+1})$ can be obtained recursively from the
catenary data of lower-rank cycle matroids $M(K_m)$. This recursion is
straightforward but complicated.  Similar recursions exist for the
catenary data of other matroids (such as bicircular matroids)
constructed from complete graphs.

\begin{ex}\label{cyclematroid}  
  Consider $M(K_5)$. In $M(K_5)$, there are five copoints isomorphic
  to $M(K_4)$ and ten isomorphic to $M(K_3) \oplus M(K_2)$. The
  catenary data of $M(K_4)$ is given in Example \ref{Kfour}.  Since
  $M(K_3) \oplus M(K_2) \cong U_{2,3} \oplus U_{1,1}$, Proposition
  \ref{prop:dirsum} yields its catenary data:
  \[
  \nu(M(K_3) \oplus M(K_2);0,1,1,2) = 6,\quad \nu(M(K_3) \oplus
  M(K_2);0,1,2,1) = 3.
  \]
  We can now use Proposition \ref{catcurse} to obtain the catenary
  data of $M(K_5)$:
  \begin{align*}
    & \nu(M(K_5);0,1,2,3,4) = 5 \cdot 12 = 60,
    \\
    & \nu(M(K_5);0,1,1,4,4) = 5 \cdot 6 \, = 30,
    \\
    & \nu(M(K_5);0,1,1,2,6) = 10 \cdot 6 = 60,
    \\
    & \nu(M(K_5);0,1,2,1,6) = 10 \cdot 3 = 30.
  \end{align*}
\end{ex}

The Dowling matroids $Q_r(G)$ based on the finite group $G$ generalize
the cycle matroids $M(K_{r+1})$ (see \cite{Dowling}).  In $Q_r(G)$,
there are $[(|G|+1)^r - 1]/|G|$ copoints.  When $|G| \geq 2$ and $0
\leq m \leq r-1$, there are $\binom {r}{m} |G|^{r-m-1}$ copoints
isomorphic to $Q_m (G) \oplus M(K_{r-m})$. Thus, by Lemma
\ref{catcurse}, the catenary data of $Q_r(G)$ can be obtained from
lower-rank Dowling matroids or cycle matroids of complete graphs by a
recursion depending only on the order $|G|$. The next result follows
by induction starting with the observation that $Q_1(G)$ is isomorphic
to $U_{1,1}$.
 
\begin{prop}\label{Dowling} 
  The $\mathcal{G}$-invariant of the Dowling matroid $Q_r(G)$ depends
  only on $r$ and $|G|$.
\end{prop}

\subsection{Free extension and coextension}\label{sec:freeexcoex}
Given a matroid $M$ and an element $x$ not in $E(M)$, the \emph{free
  extension} $M+x$ is the matroid on $E(M) \cup x$ whose bases are the
bases of $M$ or the sets of the form $B\cup x$ where $B$ is a basis of
a copoint of $M$.  Equivalently,
\[
M + x = \mathrm{Trun}(M \oplus U_{1,1}),
\]
where $U_{1,1}$ is the rank-$1$ matroid on the set $\{x\}$. The
construction dual to free extension is the \emph{free coextension}
$M\times x$, defined by $M\times x = (M^*+x)^*$.  This equals the
Higgs lift $\mathrm{Lift}(M \oplus U_{0,1})$, where $U_{0,1}$ is the
rank-$0$ matroid on the set $\{x\}$.

\begin{prop} \label{prop:fext} The $\mathcal{G}$-invariant of the free
  extension $M+x$ is obtained from $\mathcal{G}(M)$ by the
  specialization
  \[
  [r_1r_2 \ldots r_n] \mapsto [1r_1r_2 \ldots r_n]^{\downarrow} +
  \sum_{j=1}^n [r_1 \ldots r_j 1 \ldots r_n]^{\downarrow},
  \]
  and the $\mathcal{G}$-invariant of the free coextension $M\times x$
  is obtained by
  \[ [r_1r_2 \ldots r_n] \mapsto [0r_1r_2 \ldots r_n]^{\uparrow} +
  \sum_{j=1}^n [r_1 \ldots r_j 0 \ldots r_n]^{\uparrow}.
  \]
\end{prop}

A freedom (or nested) matroid is obtained from a loop or a coloop by a
sequence of free extensions or additions of a coloop.  Thus, one can
recursively compute the $\mathcal{G}$-invariant of a freedom matroid
using Corollary \ref{cor:coloops} and Proposition \ref{prop:fext}.

\begin{ex}  
  The matroid $N$ in Figure \ref{fig:sameconfig} is the free extension
  of $N\del x$. The $\mathcal{G}$-invariant of $N\del x$ is
  $96\,[11100]+24\,[11010]$.  Each occurrence of $[11100]$ in
  $\mathcal{G}(N\del x)$ gives rise to six terms in the free extension
  in the following way.  First insert a $\mathbf{1}$ to obtain
  \[ [\mathbf{1}11100]+[1\mathbf{1}1100]+
  [11\mathbf{1}100]+[111\mathbf{1}00]+[1110\mathbf{1}0]+
  [11100\mathbf{1}]
  \]
  and then change the right-most $1$ to $\mathbf{0}$
  \[ [111\mathbf{0}00]+[111\mathbf{0}00]+
  [111\mathbf{0}00]+[111\mathbf{0}00]+[1110\mathbf{0}0]+
  [11100\mathbf{0}].
  \]
  Likewise, each occurrence of $[11010]$ gives rise to six terms:
  \[ [111000]+[111000]+ [111000]+[110100]+[110100]+ [110100].
  \]
  From this, we obtain
  \[
  \mathcal{G}(N) = (96\cdot 6+3\cdot 24)\,[111000]+3\cdot 24\,[110100]
  = 648\,[111000]+72\,[110100].
  \]
\end{ex}

\subsection{Free product}

The free product is a non-commutative matroid operation defined by
Crapo and Schmitt~\cite{fpm}.  Given an ordered pair $M_1$ and $M_2$
of matroids of ranks $r_1$ and $r_2$ on disjoint sets $E_1$ and $E_2$,
the \emph{free product} $M_1\frp M_2$ is the matroid on $E_1\cup E_2$
whose bases are the subsets $B$ of $E_1\cup E_2$ of size $r_1+r_2$
such that $B\cap E_1$ is independent in $M_1$ and $B\cap E_2$ spans
$M_2$. The rank function of $M_1 \frp M_2$ is given as follows: for
$X_i \subseteq E_i$,
\begin{equation}\label{eq:frprk}
  r_{M_1\frp M_2}(X_1\cup X_2)= \min\{r_1(E_1) + r_2(X_2),r_1(X_1)+|X_2|\}.
\end{equation}
The free product $U_{1,1} \frp M$ is the free
  coextension of $M$, while $M\frp U_{0,1}$ is the free extension of
  $M$. 

\begin{prop}\label{prop:Gfrpr}
  The $\mathcal{G}$-invariant of the free product $M_1\frp M_2$ can be
  calculated from the $\mathcal{G}$-invariants of $M_1$ and $M_2$.
\end{prop}

\begin{proof}
  Let $E_1 = \{1,2,\ldots, m\}$ and $E_2 = \{m+1,m+2, \ldots, m+n\}$.
  Let $[\underline{r}_i]$ be symbols that occur in $\mathcal{G}(M_i)$,
  so $\underline{r}_i$ is the rank sequence of a permutation $\pi_i$
  on $E_i$.  While $\mathcal{G}(M_i)$ does not give $\pi_i$, we can
  calculate the rank sequence of a shuffle
  $\pi=\mathrm{sh}(\pi_1,\pi_2;P)$ from the triple
  $(\underline{r}_1,\underline{r}_2,P)$.  Given $j$ with $1 \leq j
  \leq m+n$, let $X_1 = \{\pi(1),\pi(2),\ldots, \pi(j)\} \cap E_1$ and
  $X_2 = \{\pi(1),\pi(2),\ldots, \pi(j)\} \cap E_2$, so the set
  $\{\pi(1),\pi(2),\ldots, \pi(j)\}$ is the disjoint union of $X_1$
  and $X_2$. Note that $|X_1| = |P \cap \{1,2,\ldots,j\}|$ and $|X_2|
  = j - |X_1|$, and that $r_i(X_i)$ is the number of $1$'s in the
  first $|X_i|$ positions in $\underline{r}_i$. By equation
  (\ref{eq:frprk}), the rank of $\{\pi(1),\pi(2), \ldots,\pi(j)\}$ in
  $M_1 \frp M_2$ can be found from $r_1(E_1)$, $r_2(X_2)$, $r_1(X_1)$,
  and $|X_2|$, and hence, from $(\underline{r}_1, \underline{r}_2,
  P)$. Applying this procedure starting from $j=1$ yields the rank
  sequence of $\mathrm{sh}(\pi_1,\pi_2;P)$ from the triple
  $(\underline{r}_1, \underline{r}_2, P)$.

  Each permutation of $\{1,2,\ldots,m+n\}$ has exactly one
  representation as a shuffle of permutations $\pi_i$ of $E_i$, so the
  multiset of rank sequences $\underline{r}(\pi)$, over all
  permutations $\pi$ $\{1,2,\ldots,m+n\}$, can be obtained by
  calculating the multiset of rank sequences obtained from the triples
  $(\underline{r}_1, \underline{r}_2, P)$, where $\underline{r}_i$
  ranges over the multiset of rank sequences that occur in
  $\mathcal{G}(M_i)$ and $P$ ranges over all $m$-subsets of
  $\{1,2,\ldots,m+n\}$. In this way we obtain $\mathcal{G}(M_1 \frp
  M_2)$ from $\mathcal{G}(M_1)$ and $\mathcal{G}(M_2)$.
\end{proof}

\begin{ex}\label{frpranksequence}  
  Let $M_1 = U_{1,2}\oplus U_{1,1}$ and $M_2 =U_{1,3}\oplus U_{1,2}$
  on disjoint sets.  Then, $\mathcal{G}(M_1) = 4\,[110]+ 2\,[101]$ and
  $ \mathcal{G}(M_2) = 72\,[11000]+36\,[10100]+12\,[10010].  $ The
  free product $M_1\frp M_2$ is a rank-$4$ matroid on a set of size
  $8$. By the computation shown in the table below, the triple
  $(101,10010,\{3,4,5\})$ gives the rank sequence $11100010$ in
  $M_1 \frp M_2$. Repeating this procedure $2 \cdot 3 \cdot 56$ times,
  we obtain $\mathcal{G}(M_1 \frp M_2)$.

  \begin{center}
    \begin{tabular}{|c|c|c|c|c|c|}
      \hline
      $j$  \rule{0pt}{10pt} & $|X_1|$ & $|X_2|$ & $r_1(X_1)$ & $r_2(X_2)$ &
      $\min\{r_1(E_1)+r_2(X_2),r_1(X_1)+|X_2|\}$ \\
      \hline
      \hline
      $1$  \rule{0pt}{10pt} & $0$ & $1$ & $0$ & $1$ &
      $\min\{2+1,0+1\}=1$ \\
      \hline
      $2$  \rule{0pt}{10pt} & $0$ & $2$ & $0$ & $1$ &
      $\min\{2+1,0+2\}=2$ \\
      \hline
      $3$  \rule{0pt}{10pt} & $1$ & $2$ & $1$ & $1$ &
      $\min\{2+1,1+2\}=3$ \\
      \hline
      $4$  \rule{0pt}{10pt} & $2$ & $2$ & $1$ & $1$ &
      $\min\{2+1,1+2\}=3$ \\
      \hline
      $5$  \rule{0pt}{10pt} & $3$ & $2$ & $2$ & $1$ &
      $\min\{2+1,2+2\}=3$ \\
      \hline
      $6$  \rule{0pt}{10pt} & $3$ & $3$ & $2$ & $1$ &
      $\min\{2+1,2+3\}=3$ \\
      \hline
      $7$  \rule{0pt}{10pt} & $3$ & $4$ & $2$ & $2$ &
      $\min\{2+2,2+4\}=4$ \\
      \hline
      $8$  \rule{0pt}{10pt} & $3$ & $5$ & $2$ & $2$ &
      $\min\{2+2,2+5\}=4$ \\
      \hline
    \end{tabular}
  \end{center}
\end{ex}

\subsection{$q$-cone}

In his study of tangential blocks, Whittle \cite{qlifts} introduced
$q$-cones (originally called $q$-lifts) of $\GF(q)$-representable
simple matroids.

\begin{dfn}
  Let $M$ be a simple $\GF(q)$-representable rank-$r$ matroid and
  choose a representation of $M$ in the rank-$(r+1)$ projective
  geometry $\PG(r,q)$, that is, a set $E$ of points in $\PG(r,q)$ so
  that the restriction $\PG(r,q)|E$ is isomorphic to $M$.  Let
  $A \mapsto \overline{A}$ be the closure operator of $\PG(r,q)$.
  Choose a point $a$ in $\PG (r,q)$ not in the linear hyperplane
  $\overline{E}$ and let $E'$ be the union
  \[
  \bigcup_{p\in E} \overline{\{a,p\}}.
  \]
  The matroid $\PG(r,q)|E'$ is the \emph{$q$-cone} of $M$ with
  \emph{base} $E$ and \emph{apex} $a$.
\end{dfn}

A $\GF(q)$-representable matroid $M$ may have
  inequivalent representations, so different choices of $E$ may yield
  non-isomorphic $q$-cones of $M$: Oxley and Whittle \cite{nonunique}
gave examples of matroids with inequivalent representations that yield
non-isomorphic $q$-cones.  However, using a formula for the
characteristic polynomial of $q$-cones due to Kung \cite[Section
8.6]{critical}, Bonin and Qin \cite{qcones} showed that the Tutte
polynomial of a $q$-cone of $M$ can be calculated from the Tutte
polynomial of $M$ and so depends only on $M$.  We next treat a similar
result for the $\mathcal{G}$-invariant.

\begin{prop}
  Let $M$ be a simple $\GF(q)$-representable matroid.  The catenary
  data of a $q$-cone $M'$ of $M$ can be calculated from the catenary
  data of $M$.  Thus, the $\mathcal{G}$-invariant of a $q$-cone $M'$
  depends only on $\mathcal{G}(M)$ (not on $M$ or the representation).
\end{prop}

\begin{proof}
  Let $M'$ be the $q$-cone of $M$ with base $E$ and apex $a$.
  Identify $M$ with the restriction $\PG(r,q)|E$.  For a flag
  \[
  \emptyset = Y_0 \subset Y_1 \subset Y_2 \subset \cdots \subset
  Y_{r-1} \subset Y_r \subset Y_{r+1} = E'
  \]
  of $M'$, let
  \[
  X_i = \cl_{M'}(Y_i\cup a)\cap E.
  \]
  Since $E$ is closed in $M'$, the set $X_i$ is a flat of $M$.  The
  \emph{jump} of the flag $(Y_i)$ is the least $j$ with $a \in Y_j$.
  It follows that $X_{j-1} = X_j$ and that $(X_0,X_1,\ldots,
  X_{j-1},X_{j+1}, \ldots , X_{r+1})$ is a flag in $M$, which we call
  the \emph{projection of $(Y_i)$ onto $M$.}

  Given a flag $(X_i)$ in $M$ with composition $(a_0,a_1,\ldots,a_r)$,
  we obtain all flags $(Y_i)$ in $M'$ with projection $(X_i)$ and jump
  $j$ as follows.  For a fixed ordered basis $(b_1,b_2,\ldots,b_r)$ of
  $M$ with $X_i = \cl_M(\{b_1,b_2, \ldots,b_i\})$, choose a sequence
  $b'_1,b'_2,\ldots,b'_{j-1}$ with $b'_i \in \overline{\{a,b_i\}}-a$,
  and set
  \[
  Y_i = \left\{ \begin{array}{ll} \cl_{M'}(\{b'_1,b'_2, \ldots,b'_i\})
      & \,\,\text{if }\, 0 \leq i \leq j-1,
      \\
      \cl_{M'}(X_{i-1} \cup a) & \,\, \text{if }\, j \leq i \leq r+1.
    \end{array}\right. 
  \]
  It is easy to check that all flags in $M'$ with projection $(X_i)$
  and jump $j$ arise exactly once this way, and, due to the choice of
  $b'_1,b'_2,\ldots,b'_{j-1}$, there are $q^{j-1}$ such flags.  Also,
  \[|Y_i| = \left\{ \begin{array}{ll} |X_i| & \,\,\mathrm{if} \,\,0
      \leq i \leq j-1,
      \\
      q|X_{i-1}|+1 & \,\, \mathrm{if}\,\, j \leq i \leq r+1,
    \end{array}\right. 
  \]
  so a flag $(X_i)$ in $M$ with composition $(a_0,a_1,\ldots,a_r)$ is
  the projection of $q^{j-1}$ flags in $M'$ with jump $j$, and these
  flags all have composition
  \[
  (a_0,a_1,\ldots,a_{j-1},(a_0 + a_1 + \cdots + a_{j-1})(q-1)+1, a_j
  q, a_{j+1}q,\ldots,a_rq).
  \]
  Thus, we get the catenary data of $M'$ from that of $M$.
\end{proof} 

\begin{ex} 
  The matroids $M$ and $N$ in Figure \ref{fig:sameconfig} are
  representable over the field $\GF(q)$ whenever $q$ is a prime power
  with $q>3$. They have the same catenary data, so their $q$-cones
  have the same catenary data.  The catenary data of $M$ is
  $\nu(M;0,1,2,3)=6$ and $\nu(M;0,1,1,4)=18$. Each of the six flags
  with composition $(0,1,2,3)$ is the projection of
  \begin{enumerate}
  \item one flag with jump $1$ and composition $(0,1,q,2q,3q)$,
  \item $q$ flags with jump $2$ and composition $(0,1,q,2q,3q)$,
  \item $q^2$ flags with jump $3$ and composition $(0,1,2,3q-2,3q)$,
  \item $q^3$ flags with jump $4$ and composition $(0,1,2,3,6q-5)$.
  \end{enumerate}
  This and a similar calculation for the other $18$ flags of $M$ give
  the catenary data for a $q$-cone $M'$:
  \begin{align*}
    \nu(M';0,1,q,2q,3q)=\, 6(1+q),\,\, & \,\,\nu(M';0,1,q,q,4q)=\, 18(1+q),\\ 
    \nu(M';0,1,2,3q-2,3q)=\, 6q^2,\,\, & \,\,\nu(M';0,1,1,2q-1,4q)=\,18q^2,\\
    \nu(M';0,1,2,3,6q-5)=\, 6q^3, \,\, & \,\, \nu(M';0,1,1,4,6q-5)=\, 18q^3.
  \end{align*}
\end{ex}

\subsection{Circuit-hyperplane relaxation}

Recall that if $X$ is a both a circuit and a hyperplane
(that is, copoint) of $M$, then the corresponding
\emph{circuit-hyperplane relaxation} is the matroid on the same set
whose bases are those of $M$ along with $X$.  The next result is easy
and extends a well-known result about Tutte polynomials.

\begin{prop}
  If $M'$ is obtained from the rank-$r$ matroid $M$ on an $n$-element
  set by relaxing a circuit-hyperplane, then
  $$\mathcal{G}(M')-\mathcal{G}(M) = r!\,(n-r)!\,([1^r0^{n-r}] - [1^{r-1}
  010^{n-r-1}]).$$
  Equivalently, the only compositions for which $\nu(M)$ and $\nu(M')$
  differ are the following:
  $$\nu(M';0,1,\ldots,1,n-r+1) = \nu(M;0,1,\ldots,1,n-r+1)+r!$$
  and
  $$\nu(M';0,1,\ldots,1,2,n-r) = \nu(M;0,1,\ldots,1,2,n-r)-\frac{r!}{2}.$$
\end{prop}

\section{Deriving matroid parameters from the
  $\mathcal{G}$-invariant}\label{sec:parameters}

Since the Tutte polynomial is a specialization of the
$\mathcal{G}$-invariant, any matroid parameter that is derivable from
the Tutte polynomial is derivable from the $\mathcal{G}$-invariant or
catenary data.  An easy but important example is the number $b(M)$ of
bases of a matroid $M$.  Indeed, the coefficient of the maximum symbol
$[1^r0^{n-r}]$ in $\mathcal{G}(M)$ equals $r!(n-r)!b(M)$.
Alternatively, by Lemma \ref{lem:countbybases},
$$b(M) =
\frac{1}{r!}\sum_{(a_0,a_1,\ldots,a_r)}\nu(M;a_0,a_1,\ldots,a_r)\,\,
a_1a_2\cdots a_r.$$
In this section, we identify some of the parameters of a matroid that
are derivable from the $\mathcal{G}$-invariant but not from the Tutte
polynomial.

The result motivating this section is that the number $f_k(s)$ of
flats of rank $k$ and size $s$ can be derived from the
$\mathcal{G}$-invariant.  As the matroid $M_1$ in Figure
\ref{fig:2exp} has two $2$-point lines whereas $M_2$ has three, but
$T(M_1;x,y) = T(M_2;x,y)$, the Tutte polynomial does not determine all
numbers $f_k(s)$. However, for a rank-$r$ matroid and a given rank
$k$, the maximum size $m_k$ of a rank-$k$ flat and the number
$f_k (m_k)$ is derivable from $T(M;x,y)$.  Indeed, $m_k$ is the
greatest integer $m$ for which the monomial $(x-1)^{r-k} (y-1)^{m-k}$
occurs in $T(M;x,y)$ with non-zero coefficient and $f_k(m_k)$ is that
coefficient. As shown in Section 5 of \cite{MR2768784}, $f_k (s)$ is
derivable from $T(M;x,y)$ for each $s$ with $m_{k-1}<s\leq m_k$.

Let $M$ be a rank-$r$ matroid on $n$ elements.  Fix integers $h$ and
$k$ with $0 \leq h \leq k \leq r$, and let
$(s_h, s_{h+1}, \ldots, s_k)$ be a sequence of positive integers.  For
$h< j \leq k$, let $s_j^{\prime} = s_{j} - s_{j-1}$.  An
\emph{$(h,k;s_h, s_{h+1}, \ldots, s_k)$-chain} is a (saturated) chain
$(X_h, X_{h+1}, \ldots , X_k)$ of flats such that $r(X_j) = j$ and
$|X_j| = s_j$ for $h \leq j \leq k$.  For a rank-$h$ flat $X$ with
$|X|=s_h$ and a rank-$k$ flat $Y$ with $|Y|=s_k$,
let $F_{X,Y}(s_h, s_{h+1}, \ldots, s_k)$ be the
number of $(h,k;s_h, s_{h+1}, \ldots, s_k)$-chains
$(X, X_{h+1}, \ldots ,X_{k-1},Y)$ starting with the flat $X$ and
ending with the flat $Y$.  Finally, let
\[
F_{h,k} (s_h, s_{h+1}, \ldots, s_k) = \sum_{|X| = s_h, |Y| = s_k}
F_{X,Y}(s_h, s_{h+1}, \ldots, s_k),
\]
so $F_{h,k} (s_h, s_{h+1},\ldots, s_k)$ is the number of
$(h,k;s_h, s_{h+1}, \ldots, s_k)$-chains.  When $h=k$, each chain
collapses to a rank-$k$ flat, and $f_k(s) = F_{k,k}(s)$, where
$s=s_k$.

In the next lemma, we shall think of elements in the $\gamma$-basis as
variables and linear combinations in $\mathcal{G}(n,r)$ as polynomials
in those variables, so one can multiply them as polynomials.

\begin{lemma}\label{GresX} 
  Let $0 \leq h \leq k \leq r$. Then
  \begin{align*}\label{eq:GresX}
    & \sum\limits_{\substack{ X \,  \mathrm{a}\, \mathrm{flat},
       \,r(X) = h, \, |X| = 
       s_h \\ Y \,  \mathrm{a}\, \mathrm{flat}, \,r(Y) = k, \, |Y| =
    s_k}}  \mathcal{G}(M|X)
    F_{X,Y}(s_h, s_{h+1}, \ldots, s_{k-1},s_k) 
 \mathcal{G}(M/Y)   
    \\   
    & \quad = \sum_{(a_j)}  \nu (M;a_0,a_1, \ldots, a_h, s'_{h+1},
       \ldots,s'_{k}, a_{k+1},\ldots,a_r) \gamma(a_0,a_1, \ldots,
       a_h)\gamma(0,a_{k+1},\ldots,a_r),  \notag 
  \end{align*}  
  where the sum ranges over all $(n,r)$-compositions
  $$(a_0,\ldots,a_h, s'_{h+1},
  \ldots, s'_{k-1},s'_{k}, a_{k+1},\ldots,a_r)$$
  such that $a_0+a_1+\cdots+a_h=s_h$ and
  $a_{k+1}+a_{k+2}+ \cdots+a_r = n-s_k$.
\end{lemma}

\begin{proof}

  For a rank-$h$ flat $X$ with $|X|=s_h$ and a rank-$k$ flat $Y$ with
  $|Y|=s_k$, let  $$\nu(M; a_0,a_1, \ldots,a_{h-1}, X,
  s^{\prime}_{h+1},\ldots,s^{\prime}_{k-1}, Y,a_{k+1}, \ldots,a_r)$$
  be the number of chains $(X_i)$ in $M$ satisfying $r(X_j) = j$,
  $|X_0| = a_0$, $|X_j - X_{j-1}| = a_j$ for $1 \leq j < h$ and
  $k< j \leq r$, $X_h = X$, $X_k = Y$, and
  $|X_j - X_{j-1}| = s^{\prime}_j$ for $h < j \leq k$. Then
  \begin{align*}
    &  \nu(M; a_0,a_1, \ldots,a_{h-1}, X,
    s^{\prime}_{h+1},\ldots,s^{\prime}_{k-1},Y,a_{k+1}, \ldots,a_r)  
    \\
    & \qquad   = 
        \nu(M|X; a_0,a_1, \ldots,a_{h-1},a_h) F_{X,Y}(s_h, s_{h+1},
       \ldots, s_{k-1},s_k)\nu(M/Y;0,a_{k+1},\ldots,a_r).  
  \end{align*} 
  Hence,
  \begin{align*}
    &  \sum_{(a_j)}  \nu(M; a_0,a_1, \ldots, a_{h-1},X,s'_{h+1},\ldots,
    s'_{k-1},Y,a_{k+1}, \ldots,a_r)  
    \gamma(a_0,a_1, \ldots,a_h) \gamma(0,a_{k+1},\ldots,a_r) 
    \\
    & \qquad =  \mathcal{G}(M|X) F_{X,Y}(s_h, s_{h+1}, \ldots,
       s_{k-1},s_k) \mathcal{G}(M/Y). 
  \end{align*}  
  Summing over all flats $X$ and $Y$ having the stated rank and size,
  we obtain
  \begin{align*}
    & \sum_{X,Y} \mathcal{G}(M|X) F_{X,Y}(s_h, s_{h+1},
      \ldots, s_k) \mathcal{G}(M/Y)   
    \\
    & \quad = \sum_{(a_j)} \left(\sum_{X,Y} \nu(M; a_0,a_1, \ldots,
      a_{h-1},X,s^{\prime}_{h+1},\ldots,s^{\prime}_{k-1},Y,a_{k+1}, \ldots,a_r)
      \right) 
    \\ 
    & \qquad\qquad\qquad\qquad\qquad\qquad\qquad\qquad       \times
      \gamma(a_0,a_1, \ldots,a_h)    \gamma(0,a_{k+1},\ldots,a_r)  
    \\
    & \quad = \sum_{(a_j)} \nu(M; a_0,a_1, \ldots,
      a_{h-1},a_h,s^{\prime}_{h+1},\ldots,s^{\prime}_{k-1},s'_{k},a_{k+1}, \ldots,a_r) 
    \\
    & \qquad\qquad\qquad\qquad\qquad\qquad\qquad\qquad   \times 
      \gamma(a_0,a_1,\ldots, a_h)\gamma(0,a_{k+1},\ldots,a_r).     
      \qedhere
\end{align*}
\end{proof} 

\begin{prop}\label{saturatedchains}  
  The numbers $F_{h,k} (s_h, s_{h+1}, \ldots, s_k)$ and in particular,
  $f_k(s)$, are derivable from the catenary data.  Also, if
  $f_k(s) = 1$ and $F$ is the unique flat of rank $k$ and size $s$,
  then both $\mathcal{G}(M|F)$ and $\mathcal{G}(M/F)$ can be derived
  from $\mathcal{G}(M)$.
\end{prop}

\begin{proof}  
  For the first assertion, specialize all the symbols in the equation
  in Lemma \ref{GresX} to $1$. Then we have $\mathcal{G}(M|X) = s_h!$
  and $\mathcal{G}(M/Y) = (n-s_k)!$, and the sum on the left equals
  $s_h! (n-s_k)! F_{h,k}(s_h, s_{h+1}, \ldots, s_k)$. Since the sum on
  the right can be calculated from the catenary data of $M$, we can
  derive $F_{h,k}(s_h, s_{h+1}, \ldots, s_k)$.  For the second
  assertion, to get $\mathcal{G}(M|F)$, take $X = F$ and $Y=E(M)$; to
  get $\mathcal{G}(M/F)$, take $X = \cl(\emptyset)$ and $Y=F$.
\end{proof}

\begin{cor}\label{cocircuits}
  The number of cocircuits of size $s$, the number of circuits of size
  $s$, and the number of cyclic sets (that is, unions of circuits) of
  size $s$ and rank $j$ are derivable from the catenary
  data.  In particular, one can deduce whether the
    matroid has a spanning circuit, so one can determine whether a
    graph is Hamiltonian from the $\mathcal{G}$-invariant of its cycle
    matroid. 
\end{cor}

\begin{proof}
  A cocircuit is the complement of a copoint; hence, the number of
  cocircuits of size $t$ equals $f_{r-1} (n-t)$. Circuits in $M$ are
  cocircuits in the dual $M^*$. Finally, a set is cyclic if and only
  if it is a union of cocircuits in $M^*$, that is, its complement is
  a flat of $M^*$.
\end{proof}

The matroids in Figure \ref{fig:2exp} show that none of the parameters
in Corollary \ref{cocircuits} is determined by the Tutte polynomial.

Proposition \ref{saturatedchains} can be used to derive $f_k(s,c)$,
the number of rank-$k$ flats $X$ of size $s$ such that the restriction
$M|X$ has exactly $c$ coloops.  The number $f_k(s,0)$ is the number of
\emph{cyclic} flats (that is, flats without coloops) of rank $k$ and
size $s$. We use the following easy lemma.

\begin{lemma}\label{ones} 
  Let $(X_0,X_1,\ldots,X_r)$ be a flag with composition
  $(a_0,a_1,\ldots,a_r)$. Then the restriction $M|X_{i+1}$ is the
  direct sum of $M|X_i$ and a coloop if and only if $a_{i+1}=1$.
\end{lemma}

\begin{prop}\label{prop:recovernumbercoloops} 
  The number $f_k(s,c)$ is derivable from the $\mathcal{G}$-invariant.
\end{prop}

\begin{proof} 
  For $j$ with $0\leq j\leq k$, the numbers $f_k(s,j)$ satisfy the
  linear equations, one for each $c$ with $0\leq c\leq k$,
  \[
  \sum_{j=c}^k f_{k}(s,j) \frac {j!}{(j-c)!} = F_{k-c,k}(s-c,
  s-c+1,s-c+2, \ldots, s),
  \]
  where the sequence $s-c,s-c+1,s+c+2,\ldots,s$ increases by $1$ at
  each step.  To see this, note that the chains that
  $F_{k-c,k}(s-c, s-c+1, \ldots, s)$ counts are obtained by picking a
  rank-$k$ size-$s$ flat that has exactly $j$ coloops with
  $c\leq j\leq k$, and going down from rank $k$ to rank $k-c$ by
  deleting $c$ of the $j$ coloops one by one in some order.  As the
  system of linear equations is triangular with diagonal entries equal
  to $c!,$ we can solve it to get $f_k(s,c).$
\end{proof} 

Much of the interest in the $\mathcal{G}$-invariant has centered on
the fact that it is a universal valuative invariant, so we end this
section by relating our work to that part of the theory.  For brevity,
we address these remarks to readers who are already familiar with
valuative invariants as discussed in \cite{Derksen}.  We show that the
parameters studied in this section are valuative invariants.

To make explicit the dependence on the matroid $M$, we shall write,
for example, $f_k(M;s)$ instead of $f_k(s)$. We shall use two results
from the theory of valuative invariants.  The first is the basic
theorem of Derksen \cite{Derksen} that specializations of the
$\mathcal{G}$-invariant, in particular, the $\mathcal{G}$-invariant
coefficients $g_{\underline{r}}(M)$, are valuative invariants.  The
second is the easy lemma that if $u$ and $v$ are valuative invariants
on size-$n$ rank-$r$ matroids, then so is any linear combination
$\alpha_{nr} u + \beta_{nr} v,$ where $\alpha_{nr}$ and $\beta_{nr}$
depend only on $n$ and $r$.
 
\begin{thm}\label{valuative}
  The following parameters are valuative invariants:
  \[
  \nu(M;a_i),\, F_{h,k}(M;s_i),\, f_k(M;s), \, \text{and} \,\,
  f_k(M;s,c).
  \]
\end{thm}

\begin{proof}   
  The catenary data is obtained from the $\mathcal{G}$-invariant by a
  change of basis in $\mathcal{G}(n,r)$.  Hence, $\nu(M;a_i)$ is a
  linear combination of $g_{\underline{r}}(M)$; explicitly,
  $\nu(M;a_i)$ is obtained from the vector of $\mathcal{G}$-invariant
  coefficients by applying the inverse of
  $[c_{(a_0,a_1,\ldots,a_r)}(b_0,b_1, \ldots,b_r) ]$, a matrix
  depending only on $n$ and $r$.  Hence, $\nu(M;a_i)$ is a valuative
  invariant.  In the notation of Lemma \ref{GresX}, the number
  $F_{h,k}(M;s_h,\ldots,s_k)$ is given in terms of the catenary data
  by
  \[
  \frac {1}{s_h! (n-s_k)!} \sum_{(a_j)} \gamma_1(a_0,\ldots,a_h)
  \gamma_1(0,a_{k+1},\ldots,a_r)
  \nu(M;a_0,\ldots,a_{h},s_{h+1}^{\prime},\ldots,s_k^{\prime},a_{k+1},\ldots,a_r),
  \]
  where the numbers $\gamma_1(a_0,\ldots,a_h)$ and
  $\gamma_1(0,a_{k+1},\ldots,a_r)$ are obtained from
  $\gamma(a_0,\ldots,a_h)$ and $\gamma(0,a_{k+1},\ldots,a_r)$ by
  specializing all symbols to $1$, and depend only on $h$, $k$, $s_h$,
  $s_k$, $n$, and $r$.  Hence $F_{h,k}(M;s_i)$ and $f_k(M;s)$ are
  valuative invariants.  Finally, the numbers $f_k(M;s,c)$ are
  obtained from the numbers $F_{h,k}(M;s_i)$ by solving a system of
  equations with coefficients depending only on $k$, $s$, and $c$;
  hence, they are valuative invariants as well.
\end{proof}

\section{Reconstructing the $\mathcal{G}$-invariant from
  decks} \label{sec:recon}

Let $\bar{\mathcal{G}}(n,r)$ be the subspace of $\mathcal{G}(n,r)$
that is spanned by the $\gamma$-basis elements that are indexed by
compositions that start with $0$. The \emph{circle product} $\odot$ is
the binary operation from $\mathcal{G}(n_1,r_1) \times
\bar{\mathcal{G}}(n_2,r_2)$ to $\mathcal{G}(n_1+n_2, r_1 + r_2)$
defined by
\[
\gamma (a_0,a_1, \ldots,a_{r_1}) \odot \gamma(0,b_1, \ldots,b_{r_2})  = 
\gamma(a_0,a_1, \ldots, a_{r_1}, b_1,b_2,\ldots,b_{r_2})
\]  
on $\gamma$-basis elements and extended to $\mathcal{G}(n_1,r_1)
\times \bar{\mathcal{G}}(n_2,r_2)$ by bilinearity.  Let
$\mathcal{F}_k$ denote the set of all rank-$k$ flats in a matroid $M$.
The next lemma shows that only the sets $\mathcal{F}_k$ have a
property that is a key to the work in this section.

\begin{lemma}
  For a set $\mathcal{A}$ of flats of a matroid $M$, the following
  statements are equivalent:
  \begin{enumerate}
  \item $\mathcal{A} = \mathcal{F}_k$ for some $k$ with $0\leq k\leq
    r(M)$, and
  \item each flag of $M$ contains exactly one flat in $\mathcal{A}$.
  \end{enumerate}
\end{lemma}

\begin{proof}
 It is immediate that statement (1) implies
    statement (2).  Now assume that statement (2) holds.  We claim
  that statement (1) holds where $k$ is
  $\max\{r(X)\,:\,X\in \mathcal{A}\}$.  A routine exchange argument
  shows that if $S$ and $T$ are distinct flats in $\mathcal{F}_k$,
  then there is a sequence $S,U,\ldots,T$ of flats in $\mathcal{F}_k$
  such that the intersection of each pair of consecutive flats has
  rank $k-1$.  If $\mathcal{A}\ne\mathcal{F}_k$, then from such a
  sequence with $S\in \mathcal{F}_k\cap \mathcal{A}$ and
  $T\in \mathcal{F}_k-\mathcal{A}$, we get flats
  $X\in \mathcal{F}_k\cap \mathcal{A}$ and
  $Y\in \mathcal{F}_k-\mathcal{A}$ with $r(X\cap Y)=k-1$.  Extend a
  flag $(Z_i)$ of $M|(X\cap Y)$ to a flag $(X_i)$ of $M$ that contains
  $X$, and, separately, to a flag $(Y_i)$ of $M$ that contains $Y$.
  Observe that at least one of $(X_i)$ or $(Y_i)$ contradicts
  statement (2); thus, statement (1) holds.
\end{proof}

\begin{thm}[The slicing formula]\label{slice} 
  Let $M$ be a rank-$r$ matroid $M$.  For $k$ with $0\leq k\leq r$,
  \[
  \mathcal{G}(M) = \sum_{X \in \mathcal{F}_k} \mathcal{G}(M|X) \odot
  \mathcal{G}(M/X).
  \]
  Thus,
  \[
  \mathcal{G}(M) = \frac {1}{r+1} \sum_{X \in L(M)} \mathcal{G}(M|X)
  \odot \mathcal{G}(M/X).
  \]
\end{thm}

\begin{proof}   
  A flat $X$ in $\mathcal{F}_k$ determines the set
  \[
  \mathcal{F}_X = \{(X_i): X_k = X\}
  \]
  of flags intersecting $\mathcal{F}_k$ at $X$. These subsets partition
  the set of all flags into $|\mathcal{F}_k|$ blocks.  Using this
  partition and Theorem \ref{thm:catdata}, we obtain  
  \begin{align*}
  &  \mathcal{G}(M) =\, \sum_{X \in \mathcal{F}_k} \Biggl(\sum_{(X_i) \in
        \mathcal{F}_X}  \gamma (|X_0|,|X_1-X_0|, \ldots, |X_r -
      X_{r-1}|)\Biggr)  \\
   & =\,  \sum_{X \in \mathcal{F}_k} \Biggl(\sum_{(X_i) \in
        \mathcal{F}_X}  \gamma (|X_0|,|X_1-X_0|, \ldots, |X -
      X_{k-1}|) \odot \gamma (0,|X_{k+1} - X|, \ldots, |X_r -
      X_{r-1}|)\Biggr)  \\
   & =\,\sum_{X \in \mathcal{F}_k} \mathcal{G}(M|X) \odot \mathcal{G}(M/X).  
    \qedhere 
  \end{align*}  
\end{proof}

\begin{cor}\label{cor:recon}
  For any $k$ with $0\leq k\leq r(M)$, the $\mathcal{G}$-invariant of
  $M$ can be reconstructed from the deck, or unlabeled multiset,
  \[
  \{(\mathcal{G}(M|X), \mathcal{G}(M/X)): \, X \in \mathcal{F}_k\}
  \]
  of ordered pairs of $\mathcal{G}$-invariants.  
\end{cor}

To get the dual result, recall that deletion and contraction are dual
operations, that is, $(M\del X)^*=M^*/X$ and $(M/X)^*=M^*\del X$ .  As
noted in the proof of Corollary \ref{cocircuits}, a set $X$ is a flat
of $M$ if and only if its complement, $E-X$, is a cyclic set of $M^*$.
Also, $r_M(X) =k$ if and only if $r(M)-k=|E-X|-r_{M^*}(E-X)$, where
the right side is the \emph{nullity} of $E-X$ in $M^*$.  If we replace
$M$ with its dual $M^*$, set $Y=E-X$, and apply Proposition
\ref{prop:dual}, we get the following reconstruction result.

\begin{cor}\label{cor:reconnullity} 
  For any $j$ with $0\leq j\leq r(M^*)$, the $\mathcal{G}$-invariant
  of $M$ can be reconstructed from the deck
  \[
  \{(\mathcal{G}(M|Y), \mathcal{G}(M/Y)): \, Y \text{ a cyclic set of
  } M \text{ of nullity }j\}.
  \]
\end{cor}

Two special cases of Corollary \ref{cor:recon} occur at the bottom and
top of the lattice $L(M)$ of flats.  Recall that the \emph{girth} of a
matroid is the minimum size of a circuit in it.  If $M$ has girth at
least $g+2$ and $X$ is a rank-$g$ flat, then $M|X$ is isomorphic to
$U_{g,g}$ and
\[
\mathcal{G}(M|X) = g! \gamma(0,\underbrace{1,1,\ldots,1}_g).
\]   
By Theorem \ref{slice}, 
\[
\mathcal{G}(M) = g! \gamma(0,1,1, \ldots,1) \odot \left( \sum_{X \in
    \mathcal{F}_g} \mathcal{G}(M/X) \right).
\] 

\begin{cor}\label{reconpoint}  
  If $M$ has girth at least $g+2$, then $\mathcal{G}(M)$ can be
  reconstructed from the deck $$\{(\mathcal{G}(M/X): \, X \in
  \mathcal{F}_g\}.$$ In particular, if $M$ is a simple matroid, then
  $\mathcal{G}(M)$ can be reconstructed from the deck
  $\{\mathcal{G}(M/x): \, x \,\,\mathrm{an} \, \mathrm{element}
  \,\mathrm{of}\, M\}$.  Dually, if $M$ has $n$ elements and each
  cocircuit of $M$ has more than $t$ elements, so each set of size
  $n-t$ spans $M$, then $\mathcal{G}(M)$ can be reconstructed from the
  deck
  $$\{(\mathcal{G}(M|Y): \, Y \text{ a cyclic set of } M \text{ with }
  |Y|=n-t \}.$$
\end{cor}  

The other special case gives another perspective on Proposition
\ref{catcurse}, which is equivalent to the corollary that we derive
next.  If $M$ has $n$ elements and $X$ is a copoint of $M$, then $M/X
= U_{1,n-|X|}$ and $\mathcal{G}(M/X) = \gamma(0,n-|X|)$. Hence we get
the following result.

\begin{cor}\label{reconcopoint}  
  The $\mathcal{G}$-invariant of $M$ can be reconstructed from the
  number $n$ of elements in $M$ and the deck of
  $\mathcal{G}$-invariants of copoints.
\end{cor}

Corollary \ref{reconcopoint} is motivated by the theorem of Brylawski
\cite{Brylawskirecon} that the Tutte polynomial is reconstructible
from the deck of unlabeled restrictions to copoints.  With the next
lemma, we can strengthen Corollary \ref{reconcopoint}.

\begin{lemma}\label{Reconstructn}
  The number $n$ of elements of $M$ can be reconstructed from the deck
  of $\mathcal{G}$-invariants of copoints.
\end{lemma}

\begin{proof} 
  The proof is an adaptation of Brylawski's argument.  We begin with
  an identity.  Using the notation in Section \ref{sec:parameters},
  \[
  F_{h,k-1}(s_h,s_{h+1},\ldots,s_{k-1})(n-s_{k-1})= \sum_{s_k}
  F_{h,k}(s_h,s_{h+1},\ldots,s_{k-1},s_k)(s_k-s_{k-1}).
  \]
  This identity holds because both sides equal the number of triples
  \[
  ((X_h,X_{h+1},\ldots,X_{k-1}), x, X),
  \]
  where $|X_i| = s_i$, the element $x$ is not in $X_{k-1}$, and $X =
  \cl (X_{k-1} \cup x)$. In particular,
  \[
  f_h(s_h)(n-s_h) = \sum_{s_{h+1}} F_{h,h+1}(s_h,s_{h+1})(s_{h+1}
  -s_h).
  \]
  Iterating the identity, we obtain
  \begin{eqnarray*}
    f_h(s_h) &=& \sum_{s_{h+1}}  F_{h,h+1}(s_h,s_{h+1})\frac {(s_{h+1} -s_h)}{(n-s_h)} 
    \\
    &=& \sum_{s_{h+1},s_{h+2}}  F_{h,h+2}(s_h,s_{h+1},s_{h+2})\frac
    {(s_{h+1} -s_h)(s_{h+2}-s_{h+1})}{(n-s_h)(n-s_{h+1})}   \\ 
    &\vdots& \\
    &=& 
    \sum_{s_{h+1},s_{h+2},\ldots,s_{r-1}}
    F_{h,r-1}(s_h,s_{h+1},\ldots,s_{r-1})  \left(\prod_{j=h}^{r-2}
      \frac {s_{j+1} -s_j}{n-s_j} \right).   
  \end{eqnarray*}
  Let $s_0$ be the number of loops in $M$. Then, since $f_0(s_0) = 1$,
  we have
  \begin{equation}\label{equationforn}
    1 =  \sum_{s_1,s_2,\ldots,s_{r-1}}  
    F_{0,r-1}(s_0,s_1,s_2,\ldots,s_{r-1})  \left(\prod_{j=0}^{r-2}
      \frac {s_{j+1} -s_j}{n-s_j} \right).   
  \end{equation} 
  In addition, we also have
  \[
  F_{0,r-1}(s_0,s_1,s_2,\ldots,s_{r-1}) = \sum_{\mathrm{copoint}\,
    X,\, |X| = s_{r-1}} \nu(M|X; s_0,s_1-s_0,s_2-s_1,
  \ldots,s_{r-1}-s_{r-2}),
  \]
  and hence, we can calculate the right-hand side of equation
  (\ref{equationforn}) given the deck
  $\{\mathcal{G}(M|X): \, X \in \mathcal{F}_{r-1}\}$, and solve for
  $n$. The number $n$ of elements is a solution greater than
  $\max\{s_{r-1}\}$, the maximum number of elements in a copoint, and
  this solution is unique since the right-hand side is a strictly
  decreasing function in $n$. In other words, equation
  (\ref{equationforn}), and hence the deck, determines the number of
  elements in $M$.
\end{proof}  

Corollary \ref{reconcopoint} and Lemma \ref{Reconstructn} yield the
following result.

\begin{thm}\label{thm:reconcopoint}
  The $\mathcal{G}$-invariant of $M$ can be reconstructed from the
  deck $\{\mathcal{G}(M|X): \, X \text{ a copoint}\}$.  Dually, we can
  reconstruct $\mathcal{G}(M)$ from the deck $\{\mathcal{G}(M/Y): \, Y
  \text{ a circuit}\}$.
\end{thm}

The proof of the first assertion in Theorem \ref{thm:reconcopoint}
does not use the $\mathcal{G}$-invariants $\mathcal{G}(M|X)$
individually but the sums
\[
\mathcal{H}(M;s) = \sum_{X \,\mathrm{a}\,\mathrm{copoint}, |X| = s}
\mathcal{G}(M|X).
\]  
Hence, we have the following theorem giving the exact information
needed to reconstruct $\mathcal{G}(M)$ from a deck derived from
copoints.

\begin{thm}\label{thm:ultimate}
  The $\mathcal{G}$-invariant $\mathcal{G}(M)$ and the deck
  $\{\mathcal{H}(M;s)\}$ (consisting of the sums $\mathcal{H}(M;s)$
  that are non-zero) can be constructed from each other.
\end{thm}

\section{Configurations}\label{sec:configurations}

The main result in this section, Theorem \ref{thm:configgivescat},
extends a result by Eberhardt \cite{jens} for the Tutte polynomial
(Theorem \ref{thm:jens} below) to the $\mathcal{G}$-invariant.  We
also show that the converse of Theorem \ref{thm:configgivescat} is
false.

These results use the configuration of a matroid, which Eberhardt
defined.  Recall that a set $X$ in a matroid $M$ is \emph{cyclic} if
$X$ is a union of circuits, or, equivalently, $M|X$ has no coloops.
The set $\mathcal{Z}(M)$ of cyclic flats of $M$, ordered by inclusion,
is a lattice.  The join $X\vee Y$ in $\mcZ(M)$ is $\cl(X\cup Y)$, as
in the lattice of flats; the meet $X\land Y$ is the union of the
circuits in $X\cap Y$.  It is well-known that a matroid $M$ is
determined by $E(M)$ and the pairs $(X,r(X))$ for $X\in
\mathcal{Z}(M)$.  If deleting all coloops of a matroid $M$ yields
$M'$, then we easily get $\mathcal{G}(M)$ from $\mathcal{G}(M')$, so
we focus on matroids without coloops.  (Similarly, focusing on
matroids that also have no loops, as in the proof of Theorem
\ref{thm:configgivescat}, is justified.)  The \emph{configuration} of
a matroid $M$ with no coloops is the triple $(L,s,\rho)$ where $L$ is
a lattice that is isomorphic to $\mathcal{Z}(M)$, and $s$ and $\rho$
are functions on $L$ where if $x\in L$ corresponds to $X\in
\mathcal{Z}(M)$, then $s(x)=|X|$ and $\rho(x) = r(X)$.  The
configuration does not record the cyclic flats, and we do not
distinguish between $L$ and lattices isomorphic to it, so
non-isomorphic matroids (such as the pair in Figure
\ref{fig:sameconfig}) may have the same configuration.  Two paving
matroids of the same rank $r$ with the same number $f_{r-1}(s)$ of
size-$s$ copoints, for each $s$, have the same configuration.
Gim\'enez constructed a set of $n!$ non-paving matroids of rank $2n+2$
on $4n+5$ elements, all with the same configuration (see \cite[Theorem
5.7]{cyc}).  We now state Eberhardt's result.

\begin{thm}\label{thm:jens}
  For a matroid with no coloops, its Tutte polynomial can be derived
  from its configuration.
\end{thm}

We use Theorem \ref{thm:jens} to prove Theorem
\ref{thm:configgivescat}. (With these techniques, one can give another
proof of Theorem \ref{thm:jens} since knowing $T(M;x,y)$ is equivalent
to knowing, for each pair $i$, $j$, the number of subsets of $E(M)$ of
size $i$ and rank $j$, and an inclusion/exclusion argument like that
in the proof of Lemma \ref{thm:configgivescat}.2 shows that the
configuration gives the number of such subsets.  We use only a special
case of Theorem \ref{thm:jens}, namely, the configuration gives the
number of bases.)  We also use the following elementary lemma about
cyclic flats.

\begin{lemma}\label{lem:configminors}
  Assume that $M$ has neither loops nor coloops.  If $X$ is any cyclic
  flat of $M$, then $\mathcal{Z}(M|X)$ is the interval $[\emptyset,X]$
  in $\mathcal{Z}(M)$, and $$\mathcal{Z}(M/X)=\{F-X\,:\,F\in \mcZ(M)
  \text{ and } X\subseteq F\},$$ so the lattice $\mathcal{Z}(M/X)$ is
  isomorphic to the interval $[X,E(M)]$ in $\mathcal{Z}(M)$.  Thus,
  from the configuration of $M$, we get the configuration of any
  restriction to, or contraction by, a cyclic flat of $M$.  Likewise,
  the configuration of the truncation $\mathrm{Trun}(M)$ can be
  derived from that of $M$.
\end{lemma}

\begin{thm}\label{thm:configgivescat}
  For a matroid $M$ with no coloops, its catenary data, and so
  $\mathcal{G}(M)$, can be derived from its configuration.
\end{thm}

\begin{proof}
  As noted above, we can make the further assumption that $M$ has no
  loops.  Let $\iota(N)$ denote the number of independent copoints of
  a matroid $N$, and $b(N)$ denote its number of bases.  The proof is
  built on two lemmas.

  \begin{sublemma1}
    If, for each chain
    $\emptyset=D_0\subset D_1\subset \cdots \subset D_t=E(M)$ in
    $\mathcal{Z}(M)$, we have each $r(D_i)$ and
    $\iota(M|D_i/D_{i-1})$, then we can compute the catenary data of
    $M$.
  \end{sublemma1}

  \begin{proof}
    Let $\emptyset=X_0\subset X_1\subset \cdots \subset X_r=E(M)$ be a
    flag of flats of $M$, so $r=r(M)$.  Each flat $X_i$ is the
    disjoint union of the set $U_i$ of coloops of $M|X_i$ and a cyclic
    flat $F_i$ of $M$.  Thus, $F_0\subseteq F_1\subseteq \cdots
    \subseteq F_r$.  Since $M$ has neither loops nor coloops,
    $U_0=F_0=\emptyset=U_r$ and $F_r=E(M)$.  Let $\emptyset=D_0\subset
    D_1\subset \cdots \subset D_t=E(M)$ be the distinct cyclic flats
    among $F_0,F_1,\ldots,F_r$, and let $U=U_1\cup U_2\cup \cdots\cup
    U_{r-1}$.  Thus, $|U|=r-t$.  We call $D_0, D_1, \ldots, D_t$ and
    $U$ the cyclic flats and set of coloops of the flag.  For $i$ with
    $1\leq i\leq t$, let $W_i = U\cap (D_i- D_{i-1})$, which may be
    empty.  Fix $i$ with $1\leq i<t$.  Let $j$ be the least integer
    with $D_i\subseteq X_j$, so $U_j=X_j-D_i$.  Delete all elements of
    $U_j$ from the flats in the chain $X_0\subset
    X_1\subset\cdots\subset X_j$; the distinct flats that remain form
    a chain, and the greatest is $D_i$ and the one of rank $r(D_i)-1$
    is $D_{i-1}\cup W_i$.  Thus, $W_i$ is an independent copoint of
    $M|D_i/D_{i-1}$.

    Conversely, let $\emptyset=D_0\subset D_1\subset \cdots \subset
    D_t=E(M)$ be a chain of cyclic flats. For each $i$ with $1\leq
    i\leq t$, let $W_i$ be an independent copoint of $M|D_i/D_{i-1}$
    and set $U=W_1\cup W_2\cup\cdots \cup W_t$ (so $|U| = r-t$).
    There are flags of flats of $M$ whose cyclic flats and set of
    coloops are $D_0,D_1,\ldots,D_t$ and $U$, and we obtain all such
    flags in the following way.  Take any permutation $\pi$ of
    $U\cup\{D_1,\ldots,D_t\}$ in which, for each $i$, all elements of
    $W_i$, and all $D_k$ with $k<i$, are to the left of $D_i$.  To get
    a flag from $\pi$, successively adjoin to $\emptyset$ the elements
    in initial segments of $\pi$.

    We next find the contributions of these flags to the catenary data
    of $M$.  With the type of permutation $\pi$ described above, there
    are $t$ associated integer compositions, each with $t$ parts,
    namely, for each $i$, we have the integer composition $|W_i| =
    a_{i1}+a_{i2}+\cdots+a_{it}$ where $a_{ij}$ is the number of
    elements of $W_i$ that are between $D_{j-1}$ and $D_j$ in $\pi$.
    Thus, if $j>i$, then $a_{ij}=0$.  We call the lower-triangular
    $t\times t$ matrices $A=(a_{ij})$ the matrix of compositions of
    $\pi$.  A lower-triangular $t\times t$ matrix with non-negative
    integer entries is a matrix of compositions for some such
    permutation $\pi$ if and only if its row sums are
    $|W_1|,|W_2|,\ldots,|W_t|$.  Let $a'_j$ be the sum of the entries
    in column $j$ of $A$, which is the number of elements between
    $D_{j-1}$ and $D_j$ in $\pi$.  The composition of the flag that we
    get from any permutation $\pi$ whose matrix of compositions is $A$
    is $$\bigl(0,\underbrace{1,\ldots,1}_{a'_1},|D_1|-|W_1|,
    \underbrace{1,\ldots,1}_{a'_2},|D_2|-|W_1\cup W_2|,\ldots,
    \underbrace{1,\ldots,1}_{a'_t},|D_t|-(r-t)\bigr),$$ and the number
    of such permutations $\pi$ is $$\prod_{i=1}^t
    \binom{|W_i|}{a_{i1},a_{i2},\ldots,a_{ii}} \,a'_i!,$$ where the
    multinomial coefficient accounts for choosing the $a_{ij}$
    elements of $W_i$ that will be between $D_{j-1}$ and $D_j$, and
    the factorial accounts for permutations of the elements that are
    between $D_{i-1}$ and $D_i$.

    Thus, what the chain $D_0,D_1,\ldots,D_t$ in $\mathcal{Z}(M)$ and
    the sets $W_1,W_2,\ldots,W_t$ contribute to the catenary data can
    be found from just the $t$ numbers $|W_1|,|W_2|,\ldots,|W_t|$.
    Since $|W_i| =r(M|D_i/D_{i-1})-1$ and $W_i$ can be any independent
    copoint of $M|D_i/D_{i-1}$, the catenary data of $M$ can be
    computed from the data $r(D_i)$ and $\iota(M|D_i/D_{i-1})$, over
    all chains in $\mathcal{Z}(M)$ that include $\emptyset$ and
    $E(M)$, as Lemma \ref{thm:configgivescat}.1 asserts.
  \end{proof}

  \begin{sublemma2}
    We can compute $\iota(N)$ if we have \emph{(i)}
    $b(\mathrm{Trun}(N))$ as well as \emph{(ii)} for each chain
    $\emptyset=F_0\subset F_1\subset \cdots \subset F_p \subset E(N)$
    in $\mathcal{Z}(N)$, the numbers $b(N|F_k/F_{k-1})$, for $1\leq
    k\leq p$, and $b(\mathrm{Trun}(N/F_p))$.
  \end{sublemma2}

  \begin{proof}
    Note that $b(\mathrm{Trun}(N))$ is the number of independent sets
    of size $r(N)-1$ in $N$.  Set $\mcZ'=\{ F\in\mcZ(N) \,:\,
    0<r(F)<r(N)\}$.  For $F\in \mcZ'$, let
    $$A_F =  \{I\subset E\,:\, r(I)= |I|=r(N)-1 \text{ and } 
    \cl(I\cap F)= F \}.$$ Thus,
    $$\iota(N)=b(\mathrm{Trun}(N))-\Bigl|\bigcup_{F\in \mcZ'}A_F
    \Bigr|,$$ so by inclusion/exclusion,
    \begin{equation}\label{eq:firstpie}
      \iota(N)= b(\mathrm{Trun}(N))+\sum_{S\subseteq \mcZ',
        S\ne\emptyset}(-1)^{|S|}\Bigl|\bigcap_{F\in S}A_F \Bigr|.
    \end{equation}
    Note that if $I\in A_{F_1}\cap A_{F_2}$ for $F_1,F_2\in \mcZ'$,
    then $I\cap (F_1\vee F_2)$ spans $F_1\vee F_2$ (recall, $F_1\vee
    F_2 = \cl(F_1\cup F_2)$); thus $r(F_1\vee F_2)\leq r(I)<r(N)$, so
    $F_1\vee F_2 \in \mcZ'$, and $I\in A_{F_1\vee F_2}$; hence,
    $A_{F_1}\cap A_{F_2} =A_{F_1}\cap A_{F_2}\cap A_{F_1\vee F_2}$.
    Now assume that $F_1$ and $F_2$ are incomparable, so $F_1\vee F_2$
    properly contains both of them.  If $F_1$ and $F_2$ are in a
    subset $S$ of $\mcZ'$, and if $S'$ is the symmetric difference
    $S\triangle\{F_1\vee F_2\}$, then $(-1)^{|S|}=-(-1)^{|S'|}$ and
    $$\bigcap_{F\in S}A_F  = \bigcap_{F\in S'}A_F,$$
    so such terms could cancel in the sum in equation
    (\ref{eq:firstpie}).  Pair off such terms as follows: take a
    linear extension $\leq$ of the order $\subseteq$ on $\mcZ'$ and,
    if a subset $S$ of $\mcZ'$ contains incomparable cyclic flats, let
    $F_1$ and $F_2$ be such a pair for which $(F_1,F_2)$ is least in
    the lexicographic order that $\leq$ induces on $\mcZ'\times\mcZ'$;
    cancel the term in the sum in equation (\ref{eq:firstpie}) that
    arises from $S$ with the one that arises from $S' =
    S\triangle\{F_1\vee F_2\}$.  These cancellations leave
    \begin{equation}\label{eq:secondpie}
      \iota(N)=b(\mathrm{Trun}(N))+\sum_{\substack{\text{nonempty
            chains }   \\ 
          S\subseteq \mcZ'}}
      (-1)^{|S|}\Bigl|\bigcap_{F\in S}A_F \Bigr|. 
    \end{equation}
    Now let $S$ be a chain in this sum, say consisting of $F_1\subset
    F_2\subset \cdots \subset F_p$.  The independent sets in
    $\bigcap_{i=1}^p A_{F_i}$ are the union of a basis of $N|F_1$, a
    basis of $N|F_k/F_{k-1}$, for $1<k\leq p$, and a basis of
    $\mathrm{Trun}(N/F_p)$.  Thus, $\bigl|\bigcap_{i=1}^p
    A_{F_i}\bigr|$, and so $\iota(N)$, can be found from the data
    given in Lemma \ref{thm:configgivescat}.2.
  \end{proof}

  With these lemmas, we now prove Theorem \ref{thm:configgivescat}.
  Let $(L, s,\rho)$ be the configuration of $M$.  Chains
  $\emptyset=D_0\subset D_1\subset \cdots \subset D_t=E(M)$ in
  $\mathcal{Z}(M)$ correspond to chains
  $\hat{0}=x_0<x_1<\cdots<x_t=\hat{1}$ in $L$, where $\hat{0}$ and
  $\hat{1}$ are the least and greatest elements of $L$.  The
  configuration of $M|D_i/D_{i-1}$ consists of the interval
  $[x_{i-1},x_i]$ in $L$ along with the maps $y\mapsto
  s(y)-s(x_{i-1})$ and $y\mapsto \rho(y)-\rho(x_{i-1})$.  (For what we
  do with such minors (e.g., construct more such minors and find their
  configurations), working with the matroid is equivalent to working
  with its configuration; to make it easier to follow, in the rest of
  the proof we use the matroid.)  The only other data we need in order
  to apply Lemma \ref{thm:configgivescat}.1 is $\iota(M|D_i/D_{i-1})$.
  Lemma \ref{thm:configgivescat}.2 reduces this to data that, by
  Theorem \ref{thm:jens}, can be derived from the configuration of
  $M|D_i/D_{i-1}$ and hence from that of $M$.  Specifically, letting
  $N$ be $M|D_i/D_{i-1}$, we have its configuration, say $(L',
  s',\rho')$, and the chains in $L'$ correspond to those in
  $\mathcal{Z}(N)$, and we get the configurations of
  $\mathrm{Trun}(N)$ and each $N|F_k/F_{k-1}$ and
  $\mathrm{Trun}(N/F_p)$ in Lemma \ref{thm:configgivescat}.2.  Thus,
  we get the number of bases of $\mathrm{Trun}(N)$ and each
  $N|F_k/F_{k-1}$ and $\mathrm{Trun}(N/F_p)$ by Theorem
  \ref{thm:jens}, so, as needed, by Lemma \ref{thm:configgivescat}.2
  we get $\iota(N)$.
\end{proof}

By Proposition \ref{Dowling}, Dowling matroids of the same rank based
on groups of the same finite order have the same
$\mathcal{G}$-invariant.  Here we prove that, for $r\geq 4$, Dowling
matroids based on non-isomorphic finite groups have different
configurations; thus, the converse of Theorem \ref{thm:configgivescat}
is false.  We first recall the few facts about these matroids that we
use.  Let $G$ be a finite group and $r$ a positive integer.  The set
on which the Dowling matroid $Q_r(G)$ is defined consists of the
$r+\binom{r}{2}|G|$ elements
\begin{itemize}
\item[(P1)] $p_1,p_2,\ldots,p_r$,
and
\item[(P2)] $a_{ij}$ for each $a\in G$ and $i,j\in \{1,2,\ldots,r\}$
  with $i\ne j$, where $a_{ji}= (a^{-1})_{ij}$.
\end{itemize}
The lines (rank-$2$ flats) of $Q_r(G)$ are of three types:
\begin{itemize}
\item[(L1)] $\ell_{ij}:=\{p_i,p_j\} \cup \{a_{ij} \,:\, a \in G \}$,
  where $|\{i,j\}| = 2$,
\item[(L2)] $\{a_{ij}, b_{jk}, (ab)_{ik} \}$, where $a,b\in G$ and
  $|\{i,j,k\}| = 3$, and
\item[(L3)] $\{p_i,a_{jk}\}$ with $|\{i,j,k\}| = 3$, and
  $\{a_{hi},b_{jk}\}$ with $|\{h,i,j,k\}| = 4$.
\end{itemize} 
The set $\{p_1,p_2,\ldots,p_r\}$ is a basis of $Q_r(G)$; each element
of $Q_r(G)$ is in some line that is spanned by two elements in this
basis.  For each $t$ with $2\leq t\leq r$, the restriction of $Q_r(G)$
to the closure of any $t$-element subset of $\{p_1,p_2,\ldots,p_r\}$
is isomorphic to $Q_t(G)$ and so has $t+\binom{t}{2}|G|$ elements; if
$|G|\geq 4$ (the least $m$ for which there are non-isomorphic groups
of order $m$), then all other flats of rank $t$ are strictly smaller.

\begin{prop}\label{prop:DMdiffconfig}
  For $r\geq 4$ and finite groups $G$ and $\hat{G}$, if the Dowling
  matroids $Q_r(G)$ and $Q_r(\hat{G})$ have the same configuration, then
  $G$ and $\hat{G}$ are isomorphic.
\end{prop}

\begin{proof}
  Assume that $|G|= |\hat{G}|\geq 4$.  By the hypothesis, there is a
  lattice isomorphism $\phi: \mcZ(Q_r(G))\to \mcZ(Q_r(\hat{G}))$ that
  preserves the rank and size of each cyclic flat.  Let $\cl$ be the
  closure operator of $Q_r(G)$, and $\cl'$ that of $Q_r(\hat{G})$.
  Applying the last remark before the proposition with $t=4$ shows
  that, with relabeling if needed, we may assume that
  $\phi(\cl(\{p_1,p_2,p_3,p_4\}))= \cl'(\{p_1,p_2,p_3,p_4\})$.  By
  restricting to these flats, we may assume that $r=4$.  Likewise, we
  may assume that $\phi(\cl(\{p_i,p_j,p_k\}))=\cl'(\{p_i,p_j,p_k\})$
  whenever $\{i,j,k\}\subset \{1,2,3,4\}$.  We will show that $Q_3(G)$
  and $Q_3(\hat{G})$ are isomorphic.  The following result by Dowling
  \cite[Theorem 8]{Dowling} then completes the proof: for $r\geq 3$,
  the matroids $Q_r(G)$ and $Q_r(\hat{G})$ are isomorphic if and only if
  the groups $G$ and $\hat{G}$ are isomorphic.

  The singleton flats $\{a_{ij}\}$ are not in $\mcZ(Q_4(G))$, but we
  show that they induce certain partitions of sets of $3$-point lines
  in $\mcZ(Q_4(G))$.  Let $i,j,s,t$ be $1,2,3,4$ in some order.  Let
  $P_{ijs}$ be the set of $3$-point lines in the plane
  $\cl(\{p_i,p_j,p_s\})$.  For $a\in G$, set
  $$P^{a_{ij}}_s = \{\ell\in P_{ijs}\,:\, a_{ij}\in \ell\}.$$  Thus,
  $\{P^{a_{ij}}_s\,:\,a\in G\}$ is a partition of $P_{ijs}$.  With
  $P_{ijt}$ and $P^{b_{ij}}_t$ defined similarly, take $\ell\in
  P^{a_{ij}}_s$ and $\ell'\in P^{b_{ij}}_t$.  By the remarks above, if
  $a\ne b$, then $\cl(\ell\cup \ell')$ is $Q_4(G)$, while if $a=b$,
  then $\cl(\ell\cup \ell')$ is a cyclic flat of rank $3$.  Now
  $\cl(\ell\cup \ell')\in \mcZ(Q_4(G))$ and the same observations hold
  for $Q_4(\hat{G})$, so the isomorphism $\phi$ maps the blocks
  $P^{a_{ij}}_s$ and $P^{a_{ij}}_t$ in the partitions of $P_{ijs}$ and
  $P_{ijt}$ to their counterparts (written, for instance, as
  $\hat{P}^{c_{ij}}_s$) in $Q_4(\hat{G})$.  Thus, for
  $\{i,j\}\subset\{1,2,3,4\}$, there is a bijection $\psi_{ij}:G\to
  \hat{G}$ with $\phi(P^{a_{ij}}_s) = \hat{P}^{(\psi_{ij}(a))_{ij}}_s$
  and $\phi(P^{a_{ij}}_t) = \hat{P}^{(\psi_{ij}(a))_{ij}}_t$.  Putting
  these maps together, we have a bijection $\psi: Q_4(G)\to
  Q_4(\hat{G})$ with $\psi(p_i) = p_i$ and $\psi(a_{ij}) =
  (\psi_{ij}(a))_{ij}$ for $\{i,j\}\subset \{1,2,3,4\}$ and $a\in G$.
  A $3$-point line $\ell=\{a_{ij},b_{jk},(ab)_{ik}\}$ of $Q_4(G)$ is
  in the equivalence classes $P^{a_{ij}}_k$, $P^{b_{jk}}_i$, and
  $P^{(ab)_{ik}}_j$, so $\phi(\ell)$ is in the equivalence classes
  $\hat{P}^{\psi(a_{ij})}_k$, $\hat{P}^{\psi(b_{jk})}_i$, and
  $\hat{P}^{\psi((ab)_{ik})}_j$, and so $\phi(\ell) = \psi(\ell)$.
  Thus, since $\psi$ also clearly preserves the lines of type (L1),
  the restriction of $\psi$ to $\cl(\{p_1,p_2,p_3\})$ shows that
  $Q_3(G)$ and $Q_3(\hat{G})$ are isomorphic, so, by Dowling's result,
  $G$ and $\hat{G}$ are isomorphic.
\end{proof}

\section{Detecting free products}\label{sec:detectfrp} 

The Tutte polynomial reflects direct sums in a remarkably faithful
way.  In \cite{Merino}, Merino, de Mier, and Noy showed that the Tutte
polynomial $T(M;x,y)$ factors as $A(x,y)B(x,y)$, for polynomials
$A(x,y)$ and $B(x,y)$ over $\mathbb{Z}$, if and only if $A(x,y)$ and
$B(x,y)$ are the Tutte polynomials of the constituents in some direct
sum factorization of $M$.  By that result and Theorem \ref{GtoT}, from
$\mathcal{G}(M)$ one can deduce whether $M$ is a direct sum; however,
we do not know whether the $\mathcal{G}$-invariants of the
constituents are determined by $\mathcal{G}(M)$.  In this section we
prove a result of this type for free products.

As noted earlier, the free product $M\frp U_{0,1}$ is the free
extension of $M$, while $U_{1,1} \frp M$ is the free coextension.  The
matroids in Figure \ref{fig:sameconfig} show that the
$\mathcal{G}$-invariant cannot detect free extensions (hence,
coextensions).  Thus, we consider only \emph{proper} free products
$M_1 \frp M_2$, by which we mean that each of $M_1$ and $M_2$ has at
least two cyclic flats.  Below we reduce the problem to \emph{sharp}
free products, which are proper free products $M_1 \frp M_2$ in which
$M_1$ has no coloops and $M_2$ has no loops. A \emph{pinchpoint} of
the lattice $\mathcal{Z}(M)$ of cyclic flats of $M$ is a cyclic flat,
neither the maximum nor the minimum, that each cyclic flat either
contains or is contained in.  We now state our main result.

\begin{thm}\label{thm:freeprdetect}
  From $\mathcal{G}(M)$, one can deduce whether $M$ is a proper free
  product.  Also, for each pinchpoint $X$ of $\mcZ(M)$, the
  $\mathcal{G}$-invariants of the constituents of the corresponding
  sharp free product factorization of $M$ can be derived from
  $\mathcal{G}(M)$ and $r(X)$.
\end{thm}

We use \cite[Proposition 6.1]{freeUniqFact}, which we cite next.

\begin{lemma}\label{lemma:fpcf}
  Let $M_1$ and $M_2$ be matroids on disjoint ground sets.  Set
  $$\mcZ'= \bigl(\mcZ(M_1) -\{ E(M_1)\}\bigr) \cup \{E(M_1)\cup Y
  \,:\, Y\in\mcZ(M_2)-\{\emptyset\}\,\}.$$
  If the factorization $M_1\frp M_2$ is sharp, then $\mcZ(M)$ is
  $\mcZ'\cup \{E(M_1)\}$; otherwise it is $\mcZ'$.
\end{lemma}

We first explain a reduction that we use: from each proper
factorization of $M$, we get a sharp factorization of $M$.  Let $M$ be
the proper free product $M_1\frp M_2$.  It follows from Lemma
\ref{lemma:fpcf} that $\mcZ(M)$ has a pinchpoint.  In particular, if
the factorization is sharp, then $E(M_1)$ is a pinchpoint.  If $M_1$
has coloops, then the maximum flat in $\mathcal{Z}(M_1)$ (which is not
$E(M_1)$) is a pinchpoint of $\mathcal{Z}(M)$.  If $Y$ is any set of
coloops of $M_1$, then another proper factorization of $M$ is
$(M_1\del Y)\frp ((M|Y) \frp M_2)$, and $(M|Y) \frp M_2$ applies free
coextension to $M_2$ a total of $|Y|$ times.  If $Y$ is the set of all
coloops of $M_1$, then the factorization
$(M_1\del Y)\frp ((M|Y) \frp M_2)$ of $M$ is sharp.  Likewise, if
$M_2$ has loops, then $E(M_1)\cup \cl_{M_2}(\emptyset)$ is a
pinchpoint of $\mathcal{Z}(M)$.  If $Y\subseteq \cl_{M_2}(\emptyset)$,
then another proper factorization of $M$ is
$(M_1\frp (M_2|Y))\frp (M_2\del Y)$, and $M_1\frp (M_2|Y)$ applies
free extension to $M_1$ a total of $|Y|$ times.  If
$Y= \cl_{M_2}(\emptyset)$, then the factorization
$(M_1\frp (M_2|Y))\frp (M_2\del Y)$ of $M$ is sharp.  Thus, finding
the sharp factorizations of $M$ is the crux of the problem.  The next
lemma (which recasts \cite[Theorem 6.3]{freeUniqFact}) relates this
more precisely to the pinchpoints of $\mcZ(M)$.

\begin{lemma}\label{cutpoint} 
  A matroid $M$ is a proper free product if and only if the lattice
  $\mcZ(M)$ has a pinchpoint.  Furthermore, the map
  $X\mapsto (M|X)\frp (M/X)$ is a bijection from the set of
  pinchpoints of $\mcZ(M)$ onto the set of sharp free product
  factorizations of $M$.
\end{lemma}

We now prove the main result of this section.

\begin{proof}[Proof of Theorem \ref{thm:freeprdetect}]
  Recall from Section \ref{sec:parameters} that from $\mathcal{G}(M)$
  we can deduce the number $f_k(s)$ of flats, and the number
  $f_k(s,0)$ of cyclic flats, of rank $k$ and size $s$ in $M$.  Set
  $$C= \{k\,:\, 0 < k < r(M) \, \text{ and } \,\sum_{s}f_k(s,0) =1\},$$ so
  $k\in C$ if and only if $0 < k < r(M)$ and exactly one cyclic flat
  of $M$ has rank $k$.  The rank of any candidate pinchpoint is in
  $C$, so if $C=\emptyset$, then $M$ is not a proper free product.
  When $C\ne\emptyset$, we test each $k\in C$ as follows.

  Let $X$ be the unique cyclic flat of rank $k$.  Set $s_0=|X|$. If
  $f_k(s_0)>1$, then there is a non-cyclic flat $F$ of rank $k$ and
  size $s_0$.  The flat $Y$ obtained from $F$ by deleting the coloops
  of $M|F$ is cyclic, and $s_0-|Y|=|F-Y| = k-r(Y)$.  If $Y \subset X$,
  then comparing rank and size shows that $M|X$ has coloops (the
  elements of $X-Y$), contrary to $X$ being cyclic.  Hence,
  $Y \not\subset X$, so $X$ is not a pinchpoint.  Thus, if $X$ is a
  pinchpoint, then $f_k(s_0)=1$.

  Now assume that $X$ is the only flat of rank $k$ and size $s_0$.  By
  Proposition \ref{saturatedchains}, from $\mathcal{G}(M)$, we get
  $\mathcal{G}(M|X)$ and $\mathcal{G}(M/X)$.  Thus, by Proposition
  \ref{prop:recovernumbercoloops} we can derive the number of cyclic
  flats in $M|X$ and in $M/X$.  Now $X$ is a pinchpoint of $M$ if and
  only if (i) the number of cyclic flats in $M|X$ equals the number of
  cyclic flats of rank at most $k$ in $M$, and (ii) the number of
  cyclic flats in $M/X$ equals the number of cyclic flats of rank at
  least $k$ in $M$.  Thus, we can detect pinchpoints, and for any
  pinchpoint $X$, we can derive $\mathcal{G}(M|X)$ and
  $\mathcal{G}(M/X)$.
\end{proof}

\section{Erratum} 

The description of the $\mathcal{G}$-invariant of the truncation of a
matroid given in Proposition \ref{prop:trun} is incorrect.  The
correct statement is given in the following proposition.

\begin{prop}
  Let $M$ be a matroid having positive rank $r$.  For any
  $(n,r-1)$-composition $(a_0,a_1,\ldots,a_{r-2},a'_{r-1})$, the
  number of flags of $\mathrm{Trun}(M)$ that have that composition,
  that is,
  $\nu (\mathrm{Trun}(M); a_0, a_1, \ldots, a_{r-2},a'_{r-1} )$, is
  $$\sum\limits_{\substack{a_{r-1},a_r\in\mathbb{N} \,:\\    a_{r-1}+a_r=a'_{r-1}}}
  \nu (M ; a_0, a_1, \ldots, a_{r-2},a_{r-1}, a_r
  )\cdot\frac{a_{r-1}}{a_{r-1} + a_r}.$$ Thus,
  $$\mathcal{G}(\mathrm{Trun}(M )) = \sum
  \nu (M ; a_0, a_1, \ldots, a_r ) \cdot\frac{a_{r-1}}{a_{r-1} + a_r}
  \cdot\gamma(a_0, a_1, \ldots , a_{r-2}, a_{r-1} + a_r ),$$ where the
  sum is over all $(n,r)$-compositions $(a_0, a_1, \ldots, a_r )$.
\end{prop}

\begin{proof}
  Let $P$ be the set of pairs $((X_i),e)$ where
  $(X_i)=(X_0,X_1,\ldots,X_{r-2},E(M))$ is a flag of
  $\mathrm{Trun}(M)$ with
  composition $(a_0,a_1,\ldots,a_{r-2},a'_{r-1})$, and
  $e\in E(M)-X_{r-2}$.  Thus,
  $|P|=\nu (\mathrm{Trun}(M); a_0, a_1, \ldots, a_{r-2},a'_{r-1}
  )\cdot a'_{r-1}$.  Now
  $$\{\cl_M(X_{r-2}\cup e)-X_{r-2}\,:\, e\in E(M)-X_{r-2}\}$$ is
  a partition of the set $E(M)-X_{r-2}$, so we can also get the pairs
  in $P$ by picking a flag of $M$ whose composition has the form
  $(a_0,a_1,\ldots,a_{r-2},*,*)$ and choosing an element that is in
  the largest two flats of the flag and in none of the smaller flats.
  Thus,
  $$|P|=\sum\limits_{\substack{a_{r-1},a_r\in\mathbb{N} \,:\\
      a_{r-1}+a_r=a'_{r-1}}} 
  \nu (M ; a_0, a_1, \ldots,a_{r-2},a_{r-1}, a_r )\cdot a_{r-1}.$$
  Setting the two expressions for $|P|$ equal to each other and
  dividing by $a'_{r-1}$, which is $a_{r-1}+a_r$, gives the result.
\end{proof}

\vspace{5pt}

\begin{center}
 \textsc{Acknowledgments}
\end{center}

\vspace{3pt}

We thank the referees for their careful reading and constructive
comments.

\end{document}